\documentclass[a4paper,two-side,10pt]{amsart}
\usepackage{amsmath,amsfonts,amssymb,amsthm} 
\usepackage[latin1]{inputenc}
\usepackage[T1]{fontenc}
\usepackage{bbm}
\usepackage{amsmath}
\usepackage{amsthm}
\usepackage{latexsym}
\usepackage{amssymb}
\usepackage[all]{xy}
\usepackage{calc}
\usepackage{multicol}
\usepackage{hyperref}
\usepackage[active]{srcltx}
\usepackage{xargs}
\newtheorem{thm}{Theorem}[section]
\newtheorem{prop}{Proposition}[section]

\newtheorem{lem}{Lemma}[section]

\newcommand{\R}{\mathbb{R}}
\numberwithin{equation}{section}
\newcommand{\N}{\mathbb{N}}

\newcommand{\be}{\begin{equation}}
\newcommand{\ee}{\end{equation}}

\newcommand{\e}{\varepsilon}

\newcounter{exercice}

\DeclareMathOperator{\Span}{span}
\DeclareMathOperator{\sign}{sign}
\usepackage{todonotes}
\newcommandx{\huom}[2][1=]{\todo[linecolor=red,backgroundcolor=red!10,bordercolor=red,#1]{#2}}

\begin{document}



\title[4NLS]{Construction of $2$-bubbles for the energy critical bi-harmonic Schr\" odinger equation}
\author[Jean-Baptiste Casteras]{Jean-Baptiste Casteras} \thanks{Corresponding author : Jean-Baptiste Casteras}
\address{CMAFCIO, Faculdade de Ci\^encias da Universidade de Lisboa, Edificio C6, Piso 1, Campo Grande 1749-016 Lisboa, Portugal}
\email{jeanbaptiste.casteras@gmail.com}

\author[Ilkka Holopainen]{Ilkka Holopainen}
\address{Department of Mathematics and Statistics, P.O. Box 68, 00014 University of
Helsinki, Finland}
\email{ilkka.holopainen@helsinki.fi}
\author[L\' eonard Monsaingeon]{L\' eonard Monsaingeon}
\address{GFM, Faculdade de Ci\^encias da Universidade de Lisboa, Edificio C6, Piso 1, Campo Grande 1749-016 Lisboa, Portugal and  IECL Universit\' e de Lorraine, F-54506, Vandoeuvre-l\`es-Nancy Cedex, France.}
\email{leonard.monsaingeon@univ-lorraine.fr}
\thanks{J.-B.C. supported by FCT - Funda\c c\~ao para a Ci\^encia e a Tecnologia, under the project: UIDB/04561/2020; I.H. supported by the Magnus Ehrnrooth foundation;  L.M. was funded by the Portuguese Science Foundation through a personal grant 2020/00162/CEECIND as well as the FCT project PTDC/MAT-STA/28812/2017}

\begin{abstract}
We construct a blowing-up solution for the energy critical focusing biharmonic nonlinear Schr\"odinger equation in infinite time in dimension $N\geq 13$. Our solution is radially symmetric and converges asymptotically to the sum of two bubbles. The scale of one of the bubble is of order $1$ whereas the other one is of order $|t|^{-\frac{2}{N-12}}$. Moreover, the phase between the two bubbles form a right angle.
\end{abstract}

\maketitle
\section{Introduction}
In this paper, we are interested in the following fourth order Schr\"odinger equation
\begin{equation}
\label{eqintro}
i\partial_t u -\Delta^2 u +\alpha \Delta u + |u|^{p} u=0,\text{ in } \R^N,
\end{equation}
where $N>4$, $\alpha \in \R$ and $0<p \leq 2^\ast -2 = \frac{8}{N-4}$.
The fourth order term in equation \eqref{eqintro} has been introduced by Karpman and
Shagalov \cite{MR1779828}. It allows to regularize and stabilize solutions to the classical
Schr\"odinger equation as observed through numerical simulations by Fibich, Ilan,
and Papanicolaou \cite{MR1898529}. We mention briefly a few selected results for \eqref{eq}.
Well-posedness has been established by Pausader \cite{MR2505703,MR2502523} (using the dispersive estimates of
\cite{MR1745182}) as well as some scattering results (we also refer to \cite{MR3462127} and to \cite{CH} for related results on Cartan-Hadamard manifolds). Recently, Boulenger and Lenzmann obtained (in)finite time blowing-up results in \cite{bou-lenz}. Standing wave, $u(t,x)=e^{i\beta t} v(x)$ for some $\beta \in \R$, and traveling wave, $u(t,x)=e^{i \beta t } v(x+t V)$ with $V\in \R^N$, solutions to \eqref{eqintro} have been studied in \cite{MR3855391,MR3976588,MR3977892,CF}.  Since we will be interested in blowing-up solutions let us only mention that under suitable assumptions radially symmetric ground state solutions $U$ were proven to be unstable by blow-up in the mass critical or super-critical regime $p\geq 8/N$ in \cite{MR4001029}. Namely, for any $\varepsilon>0$, there exists $\tilde{U} \in H^2 (\R^N)$ such that $\|\tilde{U} -U\|_{H^2}<\varepsilon$ and the solution to \eqref{eqintro} with initial data $u(0)=\tilde{U}$ blows-up in the $H^2$-norm. This result is based on the local virial inequality proved by Boulenger and Lenzmann which will be stated below.

In the following, we will focus on the pure biharmonic  Schr\" odinger equation in the energy critical regime 
\begin{equation}
\label{eq}
i\partial_t u -\Delta^2 u + |u|^{2^\ast-2} u=0,\text{ in } \R^N.
\end{equation}
We will also only work with radially symmetric functions. Before proceeding let us recall that the energy 
$$E(u)=\dfrac{1}{2}\int_{\R^N} |\Delta u|^2 dx - \dfrac{1}{2^\ast}\int_{\R^N} |u|^{2^\ast} dx,$$
is a conserved quantity by equation \eqref{eq}. We define
$$v_\lambda (x)= \dfrac{1}{\lambda^{\frac{N-4}{2}}} v \big(\frac{x}{\lambda}\big).$$
Notice that, for any $u_0\in \dot{H}^2$, we have
$$E((u_0)_\lambda )= E(u_0).$$
Equation \eqref{eq} is invariant under the same scaling in the sense that if $u(t)$ is a solution to \eqref{eq} then $U(t)= u(t_0 + \frac{t}{\lambda^2})_\lambda$ is also a solution to \eqref{eq} such that $U(0)=(u_0)_\lambda$. Observe that $E$ is defined on $\dot{H}^2$ functions thanks to the Sobolev inequality, for some constant $C$ depending on $N$, there holds
$$\|u\|_{L^{2^\ast}} \leq C \|\Delta u\|_{L^2}.$$
The extremizer $W$ of this inequality will play a fundamental role in the following. It is known that it is unique (up to scaling and translation) in $\dot{H}^2$ and solves the equation
$$\Delta^2 W = |W|^{2^\ast-2}W\text{ in } \R^N.$$
In fact, it is explicitely given by
$$W(x)= \dfrac{C_N}{(1+|x|^2)^{\frac{N-4}{2}}},$$
with $C_N= (N(N-4)(N^2 -4))^{\frac{N-4}{8}}$. It is also known (see \cite{MR1694339}) that it is non-degenerate in the sense that the kernel of the operator
$$\Delta^2\cdot - \dfrac{N+4}{N-4}W^{\frac{8}{N-4}}\cdot $$
restricted to radial functions belonging to $\dot{H}^2$ is given by $\big((N-4)/2 +x\cdot\nabla\big)W$. In this setting, the result of Boulenger and Lenzmann \cite{bou-lenz} reads as follows: suppose that either $E[u_0]<0$ or $E[u_0]\geq 0$, $E[u_0]<E[W]$, and $\|\Delta u_0\|_{L^2} > \|\Delta W\|_{L^2}$, then the solution $u\in C^0([0,T); \dot{H}^2)$ blows-up in finite time in the sense $T<\infty$ and
$$\int_0^T \int_{\R^N} |u(t,x)|^{\frac{2(N+4)}{N-4}}dx dt =\infty.$$ 
Let us also recall that global well-posedness was established by Pausader \cite{MR2505703} for initial data $u_0$ satisfying $E[u_0]<E[W]$ and $\|\Delta u_0\|_{L^2} < \|\Delta W\|_{L^2}$. 

 Our goal in this paper is to construct a global radial solutions which converges in the energy space to a sum of two bubbles $W_{\lambda_i}$. We will see that one of the bubbles blows-up in the sense that $\lambda_1 (t)\rightarrow 0$, as $t\rightarrow \infty$, whereas the second one will not, i.e. $\lambda_2 (t) \rightarrow c$ for some constant $c>0$. More precisely, our result reads as follows:

\begin{thm}
\label{main}
Let $N\geq 13$. There exists a radially symmetric solution $u:(-\infty , T_0] \rightarrow \dot{H}^2 (\R^N ; \mathbb{C} )$ of \eqref{eq} such that
$$\lim_{t\rightarrow -\infty} \|u(t)- (-i W + W_{\tilde{C} |t|^{-\frac{2}{N-12}}})\|_{\dot{H}^2}=0,$$
where $\tilde{C}$ is an explicit constant; see \eqref{1deftildec}.
\end{thm}
 We expect the same kind of result to be valid if we add a 'negative' Laplacian term to the equation i.e. for \eqref{eqintro} with $p=2^\ast -2$ and $\alpha <0$. This second order term should in fact increase the blow-up speed. On the other hand, if $\alpha >0$, we expect the blow-up speed to decrease but this might as well prevent the blow-up to occur.
Theorem \ref{main} was first obtained for the classical energy critical nonlinear Schr\" odin\-ger equation by Jendrej \cite{JJ}. Our paper owes a lot to his article, in fact, the core method of our proof is the same as his.  
We also refer to \cite{MR3904767} where they constructed the same kind of solutions for the energy critical equivariant wave maps equation and the critical wave equation. We emphasize that Theorem~\ref{main} is the first constructive blow-up result for 4NLS. We conjecture that there is a certain rigidity in our construction: if we assume that $u:(-\infty , T_0] \rightarrow \dot{H}^2 (\R^N ; \mathbb{C} )$ is a radially symmetric solution of \eqref{eq} such that
$$\lim_{t\rightarrow -\infty} \|u(t)- (e^{i\theta_1 (t)} W_{\lambda_1 (t)} +e^{i\theta_2 (t)} W_{\lambda_2 (t)})\|_{\dot{H}^2}=0,$$   
then $\theta_1 (t) \rightarrow -\frac{\pi}{2}$, $\theta_2(t) \rightarrow 0$, $\lambda_1(t) \rightarrow 1$, and $\lambda_2 (t) |t|^{\frac{2}{N-12}} \rightarrow c$ for some constant $c$. A result in this direction has been obtained in \cite{MR3842064} for the wave maps equation. A related result for the critical wave equation has been obtained by Jendrej \cite{MR3801295}. He showed that there is no radial solution decomposing asymptotically as a sum of two bubbles with opposite signs.

Our result is also related to the so-called soliton resolution conjecture. Roughly speaking, a solution bounded in an appropriate sense should decomposes asymptotically into a sum of bubbles at different scales and a solution to the linear equation called a radiation term. This conjecture is completely open for the 
Schr\"odinger equation but it was proved to hold for the radial critical wave equation very recently; see \cite{JLsoli} (which generalizes previous works of Duyckaerts, Kenig and Merle \cite{DKMfirst,MR2972605,MR4163362,MR4289254,DKMActa} proving the conjecture in odd dimension as well as \cite{MR4397184} resp. \cite{CDKM} of the same authors along with Martel resp. Collot for the dimension 4 resp. 6). In relation to this conjecture, let us point out that a solution decomposing asymptotically in finite time as a sum of a radiation plus a bubble has been constructed in \cite{MR3579128} for the critical wave equation and in \cite{MR4458521} for the critical wave maps equation. Very interestingly, in this case, the  scale of the bubble depends on the behavior of the radiation at $0$. So it is possible by playing with the radiation to obtain a continuum of blowing-up speeds. We refer to \cite{MR2494455,MR2372807} and \cite{MR3114854} for related results using a different method based on distorted Fourier transforms. We also refer to \cite{CSZ,CoteDu,JM,Xu} for related recent works on multi-bubble dynamics.

As already said the proof of Theorem \ref{main} follows closely \cite{JJ} (see also \cite{raphael2011existence}). The solution $u$ will be obtained by taking a weak limit, as $T_n\to -\infty$, of a sequence $u_n :[T_n ,T_0]\rightarrow \dot{H}^2$ of solutions to \eqref{eq} 
close to a two-bubble solution. So our main goal is to obtain uniform energy estimate on the sequence $u_n$. 

We are looking for solution to \eqref{eq} of the form
$$u(t)= e^{i\zeta (t)}W_{\mu (t)}+ e^{i\theta (t)} W_{\lambda (t)} +g(t) ,$$
for some parameters $\zeta ,\mu ,\theta,\lambda$ and a function $g$ assumed to be small in $\dot{H}^2$ norm and satisfy suitable orthogonality conditions.  More precisely, we assume that $g$ is orthogonal to the kernel of 
\[
e^{i\zeta (t)}W_{\mu (t)}+ e^{i\theta (t)} W_{\lambda (t)},
\] 
that is, to $e^{i\zeta}W_\mu ,\ e^{i\theta }W_\lambda ,\ i\e^{i\zeta} \Lambda W_\mu$, and $ie^{i\theta} \Lambda W_\lambda$, where 
\[
\Lambda v:= -\dfrac{\partial}{\partial\lambda}\Big|_{\lambda=1} (v_\lambda).
\] 
To determine the value of the different parameters, we perform a modulation analysis, namely we differenciate with respect to time the orthogonality relations satisfied by $g$. This way, we obtain a system of four equations involving the four parameters we want to determine. Intuitively, we first determine $\lambda (t)$ (we assume that $\lambda (t) \ll \mu (t)$) such that $\lambda^\prime (t) \approx \lambda (t)^{\frac{N-4}{2}}$, i.e. $\lambda (t)\approx |t|^{-\frac{2}{N-12}} $. More precisely, we find that $\lambda(t)$ has to satisfy \eqref{mode8}. We point out that this relation involves the parameters of the two bubbles as well as the fact that their phases form a 'right angle' i.e. differ by a factor $i$. We already see at this point that the condition $N\geq 13$ appears. We should also be able to deal with the case $N=12$ but in this case, $\lambda(t)$ should decay exponentially like $e^{-\frac{\|W\|_{L^{(N+4)/(N-4)}}^{(N+4)/(N-4)}}{2\|W\|_{L^2}^2}t}$ (see \eqref{mode8}). In view of the estimate of the energy \eqref{propenere2}, we choose $|\zeta +\frac{\pi}{2}| \approx |\mu -1|\approx |\theta|^3$. From a technical point of view, we establish bounds on the parameters under some bootstrap assumptions. In order to close the estimates, we need to improve these bounds. The main difficulty is to improve the estimate of $g$. To do so, we use a method which consists in expanding the energy of $E(u)$ and using the coercivity of the energy near a bubble. The main step to implement this strategy is to estimate $\theta (t)$. Thanks to the modulation analysis, we see that $\theta$ satisfies 
\[
\theta^\prime (t) + \frac{C_2}{2\|W\|_{L^2}^2} \theta (t)  \lambda (t)^{\frac{N-12}{2}} - \dfrac{K(t)}{2 \lambda(t)^4 \|W\|_{L^2}^2} \approx 0 
\] 
for some explicit $K$ satisfying $|K|\lesssim \|g\|_{\dot{H}^1}^2$ (see \eqref{mode9}). To close the estimate on $\theta$, we need to get rid of the term 
\[
\dfrac{K(t)}{2 \lambda(t)^4 \|W\|_{L^2}^2}.
\]
 This is done as in \cite{JJ} (see also \cite{raphael2011existence}) by introducing a localized virial correction to $\theta (t)$. Such virial was already used by Boulenger and Lenzmann to prove their blow-up results. Finally, we use the Brouwer fixed point theorem as in \cite{CMM} to deal with the linear instabilities of the flow i.e. the terms $a_i^\pm$; see Lemma \ref{lemmodeai}.

\section{Preliminaries}
In this section, we introduce notation and recall some well-known facts. We begin by recalling the definition of the energy functional
$$E(u)= \dfrac{1}{2}\int_{\R^N} |\Delta u|^2 dx - \int_{\R^N} F(u) dx,$$
where we set $F(z)= \dfrac{N-4}{2N} |z|^{\frac{2N}{N-4}}$. We also let $f(z)= |z|^{\frac{8}{N-4}} z$. A direct computation gives that
 $$f^\prime (z)z_1 = |z|^{\frac{8}{N-4}} \big(z_1 + \frac{8}{N-4} z \mathcal{R} (\frac{z_1}{z})\big).$$

  We denote by $W$ the unique (up to scaling) $\dot{H}^2 (\R^N)$ solution to $\Delta^2 u =f(u) $. It is given explicitely by
$$W=C_N \left(\dfrac{1}{1+|x|^2}\right)^{\frac{N-4}{2}},$$
where $C_N=\big(N(N-4)(N^2-4)\big)^{\frac{N-4}{8}}$. We denote by $u_\lambda$ the scaling leaving invariant this equation namely, for any function $u\in \dot{H}^2$,
$$u_\lambda (x)= \frac{1}{\lambda^{(N-4)/2}} u(\frac{x}{\lambda}).$$
We set
$$\Lambda v = - \frac{\partial}{\partial \lambda}\Big|_{\lambda =1} v_\lambda = \left(\frac{N-4}{2} +x\cdot\nabla \right) v.$$
More generally, we define
\[
\Lambda_s:=\frac{N}{2}-s+x\cdot\nabla
\]
for $s\in\R$. Hence $\Lambda=\Lambda_2$ above.
Using the result of \cite{MR1694339} ($W$ is non-degenerate), we deduce that
$$\langle g, L^- g  \rangle \geq 0,\ \ker L^- =\Span (W),$$
and
$$\ker L^+ = \Span (\Lambda W),$$
where
$$L^+ = \Delta^2  - \frac{N+4}{N-4} W^{\frac{8}{N-4}} \text{ and } L^- = \Delta^2 - W^{\frac{8}{N-4}}. $$

We denote the standard complex $L^2$ scalar product by
$$\langle v,w\rangle = \mathcal{R} \int_{\R^N} \overline{v(x)} w(x) dx$$
 for any $v,w\in L^2 (\R^N; \mathbb{C})$.
We will also denote by $\dot{H}^2_{rad}(\R^N)= \mathcal{E}$ the space of radial function belong to $\dot{H}^2 (\R^N)$. We equipped it with the usual norm
$$\|u\|_{\mathcal{E}}^2=\int_{\R^N} |\Delta u|^2 dx. $$

Let $Z_{\theta ,\lambda}$ be the linearisation of our equation at the function $e^{i\theta} W_\lambda$ multiplied by $-i$, namely
$$Z_{\theta ,\lambda}=-i \Delta^2 +i f^\prime (e^{i\theta} W_\lambda ).$$
Denoting $g_1 = \mathcal{R} g$ and $g_2 =\mathcal{I} g$, we see that
$$Z_{\theta ,\lambda }(e^{i\theta} g_\lambda)= \frac{e^{i\theta}}{\lambda^4}(L^- g_2 - i L^+ g_1)_\lambda .$$
Using this last formula and recalling that $W$ (resp. $\Lambda W$) is in the kernel of $L^-$ (resp. $L^+$), we get that
$$Z_{\theta ,\lambda} (ie^{i\theta} W_\lambda )= \frac{e^{i\theta}}{\lambda^4} (L^- W)_\lambda =0$$
and
$$Z_{\theta ,\lambda} (e^{i\theta} \Lambda W_\lambda) = \frac{e^{i\theta}}{\lambda^4} (-i L^+ \Lambda W)_\lambda =0. $$
Integrating by parts, we also have that
\begin{align*}
&\langle e^{i\theta} W_\lambda , Z_{\theta ,\lambda} (e^{i\theta} g_\lambda)\rangle = \langle e^{i\theta} W_\lambda ,  \frac{e^{i\theta}}{\lambda^4}(L^- g_2 - i L^+ g_1)_\lambda \rangle \\
& =\langle W,L^- g_2\rangle = \langle L^- W ,g_2 \rangle =0
\end{align*}
and
\begin{align}
\label{ortZ}
\langle ie^{i\theta} \Lambda W_\lambda , Z_{\theta ,\lambda} (e^{i\theta} g_\lambda)\rangle &= \langle i  e^{i\theta} \Lambda W_\lambda ,  \frac{e^{i\theta}}{\lambda^4}(L^- g_2 - i L^+ g_1)_\lambda \rangle  \nonumber\\
&=-\langle \Lambda W,L^+ g_1\rangle = -\langle L^+ \Lambda W ,g_1 \rangle =0.
\end{align}
Adapting the proof of \cite{MR2491692} (see section $7$), we can prove that there exist $Y^{(1)} , Y^{(2)} \in S$ (the space of Schwartz functions) and $\nu >0$ such that
\begin{equation}
\label{defYs}
L^+ Y^{(1)} =-\nu Y^{(2)} ,\ L^- Y^{(2)} =\nu Y^{(1)} .
\end{equation}
We can assume that $\|Y^{(1)}\|_{L^2} =\|Y^{(2)}\|_{L^2}=1 $. Next, we define
$$\alpha_{\theta ,\lambda}^+ = \dfrac{e^{i\theta}}{\lambda^4} (Y_\lambda^{(2)} +i Y_\lambda^{(1)}) \text{ and } \alpha_{\theta ,\lambda}^- =  \dfrac{e^{i\theta}}{\lambda^4} (Y_\lambda^{(2)} - i Y_\lambda^{(1)}). $$
For later purpose, notice that $\|\alpha_{\theta ,\lambda}^\pm \|_{L^{\frac{2N}{N+4}}} \lesssim 1$.
As previously, decomposing $g$ as $g=g_1 +i g_2$, we have
\begin{equation}
\label{prelime1}
\langle \alpha_{\theta ,\lambda}^+ , e^{i\theta} g_\lambda \rangle = \langle Y^{(2)} ,g_1\rangle + \langle Y^{(1)} , g_2 \rangle 
\end{equation}
and
\begin{equation}
\label{prelime2}
 \langle \alpha_{\theta ,\lambda}^- , e^{i\theta} g_\lambda \rangle = \langle Y^{(2)} ,g_1\rangle - \langle Y^{(1)} , g_2 \rangle .
\end{equation}
Recalling once more that $W\in \ker L^-$ and $\Lambda W \in \ker L^+$ and integrating by parts, we get
$$\langle W,Y^{(1)} \rangle = \frac{1}{\nu} \langle W , L^- Y^{(2)} \rangle = \frac{1}{\nu} \langle L^- W ,  Y^{(2)} \rangle =0 $$
and
$$\langle \Lambda W,Y^{(2)} \rangle =- \frac{1}{\nu} \langle \Lambda W , L^+ Y^{(1)} \rangle =- \frac{1}{\nu} \langle L^+ (\Lambda W) ,  Y^{(1)} \rangle =0 .$$
Plugging these expressions into \eqref{prelime1} and \eqref{prelime2}, we obtain
$$\langle \alpha_{\theta ,\lambda}^+ , ie^{i\theta } W_\lambda \rangle = \langle \alpha_{\theta ,\lambda}^- , ie^{i\theta } W_\lambda \rangle =0,  $$
and
$$ \langle \alpha_{\theta ,\lambda}^+ , e^{i\theta } \Lambda W_\lambda \rangle =\langle \alpha_{\theta ,\lambda}^- , e^{i\theta } \Lambda W_\lambda \rangle =0. $$
Since $Y^{(2)}\neq W $ and $\langle g ,L^- g \rangle >0$ if $g\notin \Span (W)$, we deduce that
$$\langle Y^{(1)} , Y^{(2)} \rangle = \frac{1}{\nu} \langle Y^{(2)} , L^- Y^{(2)} \rangle >0.$$
Finally, one can check that $\alpha_{\theta ,\lambda}^\pm$ are eigenfunctions of $Z^\ast_{\theta ,\lambda}$ (the adjoint operator of $Z_{\theta ,\lambda}$ with respect to our scalar product) with eigenvalues $\pm \frac{\nu}{\lambda^4}$, namely
\begin{equation}
\label{eigenvprel1}
\langle \alpha_{\theta ,\lambda}^+ , Z_{\theta ,\lambda} (e^{i\theta } g_\lambda ) = \frac{\nu}{\lambda^4} \langle \alpha_{\theta ,\lambda}^+ , e^{i\theta} g_\lambda  \rangle 
\end{equation}
and
\begin{equation}
\label{eigenvprel2}
\langle \alpha_{\theta ,\lambda}^- , Z_{\theta ,\lambda} (e^{i\theta } g_\lambda ) =- \frac{\nu}{\lambda^4} \langle \alpha_{\theta ,\lambda}^- , e^{i\theta} g_\lambda  \rangle.
\end{equation}
We finish this section with a simple lemma based on Taylor's expansions.
\begin{lem}
Let $N\geq 13$. For any $z_i \in \mathbb{C}$, $i=1,2,3$, we have
\begin{equation}
\label{fe1}
|f^\prime (z_1 +z_2)-f^\prime (z_1)|\lesssim |f^\prime (z_2)|,\ |f^\prime (z_1 +z_2)-f^\prime (z_1)|\lesssim |z_1|^{-\frac{N-12}{N-4}} |z_2|,\ if\ z_1 \neq 0,
\end{equation}
\begin{equation}
\label{fe2}
|f(z_1 +z_2) - f(z_1)|\lesssim |f^\prime (z_1)||z_2|+|f(z_2)|,
\end{equation}
\begin{equation}
\label{fe3}
|f(z_1 +z_2) - f(z_1)-f^\prime (z_1) z_2|\lesssim f(|z_2|),
\end{equation}
\begin{equation}
\label{fe4}
|f(z_1 +z_2) - f(z_1) -f^\prime (z_1) z_2|\lesssim |z_1|^{-\frac{N-12}{N-4}}|z_2|^2,
\end{equation}
\begin{equation}
\label{fe5}
|F(z_1 +z_2) - F(z_1)- \mathcal{R} (\overline{f(z_1)} z_2)|\lesssim |f^\prime (z_1)||z_2|^2+|F(z_2)|,
\end{equation}
\begin{equation}
\label{fe6}
|F(z_1 +z_2) - F(z_1) - \mathcal{R} (\overline{f(z_1)} z_2) - \mathcal{R} (\overline{f^\prime (z_1) z_2}z_2 ) |\lesssim |F(z_2)|.
\end{equation}
\end{lem}

\section{Modulation analysis}
In this section, we are going to perfom a modulation analysis. This will give us a precise idea of what should be required on our parameters. Let us recall that we are looking for a solution $u$ be of the form
\[
u(t)= e^{i\zeta (t)}W_{\mu (t)} +e^{i\theta (t)} W_{\lambda (t)} +g(t)
\] 
with $|\zeta + \pi /2|+|\mu -1|+|\theta|+\lambda +\|g\|_{\mathcal{E}} \ll 1$.
 Differentiating this expression with respect to $t$, we get 
$$\partial_t u = \zeta^\prime i e^{i\zeta} W_\mu - \frac{\mu^\prime}{\mu} e^{i\zeta}\Lambda W_\mu +\theta^\prime i e^{i\theta} W_\lambda - \frac{\lambda^\prime}{\lambda}e^{i\theta} \Lambda W_\lambda +\partial_t g.$$
On the other hand, using that $-\Delta^2 W_\mu +f(W_\mu) = -\Delta^2 W_\lambda + f(W_\lambda)=0 $, we have
$$-i\Delta^2 u +i f(u)= -i\Delta^2 g +i\big(f(e^{i\zeta} W_\mu +e^{i\theta} W_\lambda +g) - f(e^{i\zeta} W_\mu )- f(e^{i\theta }W_\lambda)\big).$$
We deduce from the two previous lines that $g$ satisfies
\begin{align*}
\partial_t g &=-i \Delta^2 g +i\big(f(e^{i\zeta} W_\mu +e^{i\theta} W_\lambda +g) - f(e^{i\zeta} W_\mu ) - f(e^{i\theta }W_\lambda)\big)\\
&-  \zeta^\prime i e^{i\zeta} W_\mu + \frac{\mu^\prime}{\mu} e^{i\zeta}\Lambda W_\mu -\theta^\prime i e^{i\theta} W_\lambda + \frac{\lambda^\prime}{\lambda} e^{i\theta}\Lambda W_\lambda  .
\end{align*}

 We set 
$$C_1 = \int_{\R^N} W^{\frac{N+4}{N-4}} dx,$$

$$C_2= \dfrac{N+4}{N-4} \int_{\R^N} W^{\frac{8}{N-4}} \Lambda W dx,$$
and
\begin{equation}\label{1deftildec}
\tilde{C}=\left(\frac{4\|W\|_{L^2}^2}{(N-12)C_1}\right)^{\frac{2}{N-12}}.
\end{equation}

We also impose that $g$ satisfies the following orthogonality conditions
\begin{equation}
\label{ortg}
\langle ie^{i\zeta} \Lambda W_\mu ,g \rangle = \langle -e^{i\zeta}  W_\mu ,g \rangle =\langle ie^{i\theta} \Lambda W_\lambda ,g \rangle = \langle -e^{i\theta}  W_\lambda ,g \rangle =0 .
\end{equation}

\begin{prop}\label{lem3.1}
Let $c>0$ be an arbitrarily small constant. Let $T_0=T_0(c)<0$ with $|T_0|$ large enough and $T<T_1 \leq T_0$. Suppose that for $T\leq t \leq T_1$, we have
\begin{equation}
\label{mode1}
|\zeta (t) +\frac{\pi}{2}| \leq |t|^{-\frac{3}{N-12}},
\end{equation}
\begin{equation}
\label{mode2}
|\mu (t) -1| \leq |t|^{-\frac{3}{N-12}},
\end{equation}
\begin{equation}
\label{mode3}
|\theta (t)| \leq |t|^{-\frac{1}{N-12}},
\end{equation}
\begin{equation}
\label{mode4}
|\lambda (t) - \tilde{C} |t|^{-\frac{2}{N-12}} |\leq 
|t|^{-\frac{5}{2(N-12)}},
\end{equation}
\begin{equation}
\label{mode5}
\|g\|_{\mathcal{E}} \leq |t|^{-\frac{N-3}{2(N-12)}}.
\end{equation}
Then, for $T\leq t\leq T_1$, there hold
\begin{equation}
\label{mode6}
|\zeta^\prime (t)|\leq c |t|^{-\frac{N-5}{N-12}},
\end{equation}
\begin{equation}
\label{mode7}
|\mu^\prime (t)| \leq c |t|^{-\frac{N-5}{N-12}},
\end{equation}
\begin{equation}
\label{mode8}
\left|\lambda^\prime (t) - \frac{C_1}{2\|W\|_{L^2}^2} \lambda(t)^{\frac{N-10}{2}}\right|\leq c |t|^{-\frac{2N-19}{2(N-12)}} ,
\end{equation}
and
\begin{equation}
\label{mode9}
\left|\theta^\prime (t) + \frac{C_2}{2\|W\|_{L^2}^2}\theta(t)\lambda(t)^{\frac{N-12}{2}} -\frac{K(t)}{2\lambda(t)^4 \|W\|_{L^2}^2} \right|\leq c |t|^{-\frac{N-11}{N-12}} ,
\end{equation}
where
$$K= - \langle  e^{i\theta}  \Lambda W_\lambda ,  f(e^{i\zeta} W_\mu +e^{i\theta} W_\lambda +g) - f(e^{i\zeta} W_\mu +e^{i\theta} W_\lambda) -f^\prime (e^{i\theta} W_\lambda +e^{i\zeta} W_\mu) g \rangle .$$
\end{prop}

Let us notice for later purpose that, since $|K|\lesssim \|g\|_{\mathcal{E}}^2$, a simple consequence of \eqref{mode8} and \eqref{mode9} is that
\begin{equation}\label{thetalambdaprime}
|\theta^\prime (t)|+
\frac{|\lambda^\prime (t)|}{\lambda(t)}\lesssim |t|^{-1}.
\end{equation}

\begin{proof}
Differentiating in time the orthogonality conditions \eqref{ortg}, we will get a linear system of the following form: 
$$\begin{pmatrix} M_{11} & M_{12} & M_{13} & M_{14}\\  M_{21} & M_{22} & M_{23} & M_{24}\\  M_{31} & M_{32} & M_{33} & M_{34}\\ M_{41} & M_{42} & M_{43} & M_{44}  \end{pmatrix} \begin{pmatrix}\mu^4 \zeta^\prime \\ \mu^3 \mu^\prime \\ \lambda^4 \theta^\prime \\ \lambda^3 \lambda^\prime \end{pmatrix} =\begin{pmatrix} B_1 \\ B_2 \\ B_3 \\ B_4 \end{pmatrix}.$$
In a first time, we will focus on the first row of this system. 

 {\bf First row.}

Differentiating $\langle ie^{i\zeta} \Lambda W_\mu ,g\rangle =0$ with respect to time, we have
\begin{align*}
0&= \frac{d}{dt} \langle ie^{i\zeta} \Lambda W_\mu ,g\rangle = - \zeta^\prime \langle e^{i\zeta} \Lambda W_\mu ,g \rangle - \frac{\mu^\prime}{\mu} \langle ie^{i\zeta} \Lambda \Lambda W_\mu ,g \rangle + \langle ie^{i\zeta} \Lambda W_\mu ,\partial_t g \rangle \\
&= \zeta^\prime (- \langle i e^{i\zeta} \Lambda W_\mu , i e^{i\zeta} W_\mu \rangle - \langle e^{i\zeta} \Lambda W_\mu ,g \rangle) +\frac{\mu^\prime}{\mu} (\langle  i e^{i\zeta} \Lambda W_\mu , e^{i\zeta} \Lambda W_\mu\rangle - \langle i e^{i\zeta} \Lambda \Lambda W_\mu ,g \rangle)\\
&+\theta^\prime \langle ie^{i\zeta} \Lambda W_\mu ,-i e^{i\theta} W_\lambda \rangle +\frac{\lambda^\prime}{\lambda} \langle ie^{i\zeta} \Lambda W_\mu , e^{i\theta} \Lambda W_\lambda \rangle \\
&+ \langle i e^{i\zeta} \Lambda W_\mu , -i \Delta^2 g +i (f(e^{i\zeta} W_\mu +e^{i\theta} W_\lambda +g) - f(e^{i\zeta} W_\mu) - f(e^{i\theta} W_\lambda)) \rangle .
\end{align*}

Noticing that $\langle -\Lambda W_\mu ,W_\mu \rangle =2 \|W_\mu\|_{L^2}^2 =2 \mu^4  \|W\|_{L^2}^2$, we infer that
\begin{align*}
M_{11} &= \mu^{-4}\big(- \langle i e^{i\zeta} \Lambda W_\mu , i e^{i\zeta} W_\mu \rangle - \langle e^{i\zeta} \Lambda W_\mu ,g \rangle\big)  =2 \|W\|_{L^2}^2 + O(\|g\|_{\mathcal{E}})\\
&= 2 \|W\|_{L^2}^2 +O\big(|t|^{-\frac{N-3}{2(N-12)}}\big),\\
M_{12}&= \mu^{-4}  \big(\langle  i e^{i\zeta} \Lambda W_\mu , e^{i\zeta} \Lambda W_\mu\rangle - \langle i e^{i\zeta} \Lambda \Lambda W_\mu ,g \rangle\big) = O(\|g\|_{\mathcal{E}})=O\big(|t|^{-\frac{N-3}{2(N-12)}}\big),\\
M_{13}&= \lambda^{-4}  \langle ie^{i\zeta} \Lambda W_\mu ,-i e^{i\theta} W_\lambda \rangle  =O(1),\\
\noalign{and}
M_{14}&= \lambda^{-4}  \langle ie^{i\zeta} \Lambda W_\mu ,e^{i\theta} \Lambda W_\lambda \rangle =O(1). 
\end{align*}
 
Consider next
$$B_1 =-  \big\langle i e^{i\zeta} \Lambda W_\mu , -i \Delta^2 g +i \big(f(e^{i\zeta} W_\mu +e^{i\theta} W_\lambda +g) - f(e^{i\zeta} W_\mu) - f(e^{i\theta} W_\lambda)\big) \big\rangle .$$
Using \eqref{ortZ}, this can be rewritten as
\begin{align*}
B_1 &= - \langle   e^{i\zeta} \Lambda W_\mu ,  f(e^{i\zeta} W_\mu +e^{i\theta} W_\lambda +g) - f(e^{i\zeta} W_\mu) - f(e^{i\theta} W_\lambda) - f^\prime (e^{i\zeta} W_\mu) g \rangle \\
&=-\langle   e^{i\zeta} \Lambda W_\mu ,  f(e^{i\zeta} W_\mu +e^{i\theta} W_\lambda +g) - f(e^{i\zeta} W_\mu +e^{i\theta} W_\lambda) - f^\prime (e^{i\zeta} W_\mu  + e^{i\theta} W_\lambda) g \rangle \\
 &-  \langle   e^{i\zeta} \Lambda W_\mu , f(e^{i\zeta} W_\mu +e^{i\theta} W_\lambda ) - f(e^{i\zeta} W_\mu )- f(e^{i\theta} W_\lambda) \rangle\\
&- \langle   e^{i\zeta} \Lambda W_\mu ,  (f^\prime (e^{i\zeta} W_\mu +e^{i\theta} W_\lambda ) - f^\prime (e^{i\zeta} W_\mu ))g \rangle\\
&= B_{11}+B_{12}+B_{13} .
\end{align*}
We will show that
\begin{align*}
|B_{11}|&= |\langle   e^{i\zeta} \Lambda W_\mu ,  f(e^{i\zeta} W_\mu +e^{i\theta} W_\lambda +g) - f(e^{i\zeta} W_\mu +e^{i\theta} W_\lambda) - f^\prime (e^{i\zeta} W_\mu + e^{i\theta} W_\lambda) g \rangle |\\
& \lesssim \|g\|_{\mathcal{E}}^2,\\
|B_{12} |&= |\langle   e^{i\zeta} \Lambda W_\mu ,  f(e^{i\zeta} W_\mu +e^{i\theta} W_\lambda ) - f(e^{i\zeta} W_\mu )- f(e^{i\theta} W_\lambda) \rangle | \lesssim \lambda^{\frac{N-4}{2}},\\
\noalign{and}
|B_{13}|&= |\langle   e^{i\zeta} \Lambda W_\mu ,  (f^\prime (e^{i\zeta} W_\mu +e^{i\theta} W_\lambda ) - f^\prime (e^{i\zeta} W_\mu ))g \rangle | \lesssim \lambda^{\frac{N-4}{4}} \|g\|_{\mathcal{E}} .
\end{align*}

Therefore, we will get the following estimate for $B_1$ 
$$|B_1|\lesssim  \|g\|_{\mathcal{E}}^2+\lambda^{\frac{N-4}{2}}+ \lambda^{\frac{N-4}{4}} \|g\|_{\mathcal{E}}\lesssim 
|t|^{-\frac{N-4}{N-12}}.$$
Let us start by the estimation of $|B_{11}|$. Since $\zeta \approx -\pi /2$ and $\theta \approx 0$, there holds $|e^{i\zeta }W_\mu +e^{i\theta} W_\lambda | \gtrsim W_\mu$. So, Taylor's expansion \eqref{fe4} gives
$$|f(e^{i\zeta} W_\mu +e^{i\theta} W_\lambda +g) - f(e^{i\zeta} W_\mu +e^{i\theta} W_\lambda) - f^\prime (e^{i\zeta} W_\mu +e^{i\theta} W_\lambda) g |\lesssim W_\mu^{- \frac{N -12}{N-4}} |g|^2  .$$
Since $|\Lambda W|\lesssim W$, using H\"older's and Sobolev's inequalities, we get
\begin{align}\label{b11}
|B_{11}|=& |\langle   e^{i\zeta} \Lambda W_\mu ,  f(e^{i\zeta} W_\mu +e^{i\theta} W_\lambda +g) - f(e^{i\zeta} W_\mu +e^{i\theta} W_\lambda) - f^\prime (e^{i\zeta} W_\mu) g) \rangle |\nonumber\\
& \leq \int_{\R^N} |W_\mu|^{\frac{8}{N-4}} |g|^2 dx \leq  \|W_\mu \|_{L^{\frac{2N}{N-4}}}^{\frac{8}{N-4}} \|g\|_{L^{\frac{2N}{N-4}}}^2 \lesssim \|g\|_{\mathcal{E}}^2  .
\end{align}
Next let us estimate $|B_{12}|$. Using a Taylor's expansion see \eqref{fe2}, we find

$$|f(e^{i\zeta} W_\mu +e^{i\theta} W_\lambda ) - f(e^{i\zeta} W_\mu ) -f(e^{i\theta} W_\lambda)  |\lesssim W_\mu^{\frac{8}{N-4}} W_\lambda +|f(W_\lambda)| . $$
Recalling that $W(x)=c_N (1+|x|^2)^{\frac{4-N}{2}}$ and $f(u)= |u|^{\frac{8}{N-4}}u$, we observe that $f(W)\in L^1 (\R^N)$. Moreover, using that $u_\lambda (x)= \lambda^{- \frac{N-4}{2}}u(x/\lambda)$, one can check that 
$$\|f(W_\lambda)\|_{L^1}= \dfrac{1}{\lambda^{\frac{N+4}{2}}} \int_{\R^N} W (x/\lambda)^{\frac{N+4}{N-4}} dx\approx \lambda^{\frac{N-4}{2}}.$$ 
Since $|\Lambda W_\mu|\lesssim 1$, this takes care of the term involving $|f(W_\lambda)|$. Concerning the one involving $W_\mu^{\frac{8}{N-4}} W_\lambda$, we will split it into two parts. First, when $|x|\leq 1$, we have
$$\|W_\lambda\|_{L^1 (|x|\leq 1)} = \lambda^{\frac{N+4}{2}} \|W\|_{L^1 (|x|\leq \lambda^{-1})} \lesssim \lambda^{\frac{N+4}{2}} \int_0^{\lambda^{-1}} r^{4-N} r^{N-1} dr \lesssim \lambda^{\frac{N-4}{2}}.$$
Using once more that $|\Lambda W_\mu|\lesssim 1$, we get the desired estimate. Finally, for $|x|\geq 1$, we have $\|W_\lambda\|_{L^\infty (|x|\geq 1)} \lesssim \lambda^{\frac{N-4}{2}}$ and $| \Lambda W_\mu|  |W_\mu|^{\frac{8}{N-4}}$ is bounded in $L^1 (\R^N)$. Combining the previous estimates, we have shown that
$$|B_{12}|\lesssim \lambda^{\frac{N-4}{2}}.$$
Finally, let us turn to the estimate of $|B_{13}|$. We will also split the integral into two parts. When $|x|\leq \sqrt{\lambda}$, observe that
$$|f^\prime (e^{i\zeta} W_\mu +e^{i\theta} W_\lambda ) - f^\prime (e^{i\zeta} W_\mu )|\lesssim W_\lambda^{\frac{8}{N-4}}.$$
So, in this region, we have
\begin{align*}
& |\langle 1_{\{|x|\leq \sqrt{\lambda}\}}   e^{i\zeta} \Lambda W_\mu ,  (f^\prime (e^{i\zeta} W_\mu +e^{i\theta} W_\lambda ) - f^\prime (e^{i\zeta} W_\mu ))g \rangle |\\
&\lesssim |\int_{\{|x|\leq \sqrt{\lambda}\}} W_\lambda^{\frac{8}{N-4}} |g|dx | \lesssim   \|W_\lambda^{\frac{8}{N-4}}\|_{L^{\frac{2N}{N+4}}(|x|\leq \sqrt{\lambda)}}  \|g\|_{L^{\frac{2N}{N-4}}}\\
&\lesssim \lambda^{\frac{N-4}{2}} \|W^{\frac{8}{N-4}}\|_{L^{\frac{2N}{N+4}} (|x|\leq 1/ \sqrt{\lambda})} \|g\|_{\mathcal{E}} \lesssim \lambda^{\frac{N+4}{4}}\|g\|_{\mathcal{E}} .
\end{align*}
On the other hand in the region $|x|\geq \sqrt{\lambda}$, using H\" older's inequality, we get
\begin{align*}
& |\langle  1_{\{|x|\geq  \sqrt{\lambda}\}} e^{i\zeta} \Lambda W_\mu ,  (f^\prime (e^{i\zeta} W_\mu +e^{i\theta} W_\lambda ) - f^\prime (e^{i\zeta} W_\mu ))g \rangle |\\
&\lesssim   \|W_\lambda \|_{L^{\frac{2N}{N-4}}(|x|\geq  \sqrt{\lambda})}  \|g\|_{L^{\frac{2N}{N-4}}}\\
&\lesssim (\int_{\lambda^{-\frac{1}{2}} }^\infty r^{-2N} r^{N-1} dr)^{\frac{N-4}{2N}}  \|g\|_{\mathcal{E}} \lesssim \lambda^{\frac{N-4}{4}}  \|g\|_{\mathcal{E}}  .
\end{align*}
This proves that
$$|B_{13}|\lesssim \lambda^{\frac{N-4}{4}} \|g\|_{\mathcal{E}} .$$

Next, let us deal with the second row. 

{\bf Second row.}

This time, we differentiate $\langle -e^{i\zeta}  W_\mu ,g\rangle =0$ with respect to time. Thus, we obtain
\begin{align*}
0&= \frac{d}{dt} \langle -e^{i\zeta} W_\mu ,g\rangle = - \zeta^\prime \langle i e^{i\zeta}  W_\mu ,g \rangle + \frac{\mu^\prime}{\mu} \langle e^{i\zeta}  \Lambda W_\mu ,g \rangle - \langle e^{i\zeta}  W_\mu ,\partial_t g \rangle \\
&= \zeta^\prime \big( \langle  e^{i\zeta}  W_\mu , i e^{i\zeta} W_\mu \rangle - \langle i e^{i\zeta}  W_\mu ,g \rangle\big) +\frac{\mu^\prime}{\mu} \big(-\langle   e^{i\zeta}  W_\mu , e^{i\zeta} \Lambda W_\mu\rangle + \langle  e^{i\zeta}  \Lambda W_\mu ,g \rangle\big)\\
&+\theta^\prime \langle e^{i\zeta}  W_\mu ,i e^{i\theta} W_\lambda \rangle +\frac{\lambda^\prime}{\lambda} \langle -e^{i\zeta}  W_\mu , e^{i\theta} \Lambda W_\lambda \rangle \\
&- \big\langle  e^{i\zeta}  W_\mu , -i \Delta^2 g +i \big(f(e^{i\zeta} W_\mu +e^{i\theta} W_\lambda +g) - f(e^{i\zeta} W_\mu) - f(e^{i\theta} W_\lambda)\big) \big\rangle .
\end{align*}

We deduce from this expression that
\begin{align*}
M_{21} &= \mu^{-4}\big( \langle  e^{i\zeta}  W_\mu , i e^{i\zeta} W_\mu \rangle - \langle i e^{i\zeta}  W_\mu ,g \rangle\big)  = O(\|g\|_{\mathcal{E}})=O\big(|t|^{-\frac{N-3}{2(N-12)}}\big),\\ 
M_{22}&= \mu^{-4}  \big(-\langle   e^{i\zeta}  W_\mu , e^{i\zeta} \Lambda W_\mu\rangle + \langle  e^{i\zeta}  \Lambda W_\mu ,g \rangle\big) =2\|W\|_{L^2}^2 + O(\|g\|_{\mathcal{E}})\\
&= 2 \|W\|_{L^2}^2 +O\big(|t|^{-\frac{N-3}{2(N-12)}}\big),\\
M_{23}&= \lambda^{-4}  \langle e^{i\zeta}  W_\mu ,i e^{i\theta} W_\lambda \rangle  =O(1),\\
\noalign{and}
M_{24}&= \lambda^{-4}  \langle -e^{i\zeta}  W_\mu ,e^{i\theta} \Lambda W_\lambda \rangle =O(1). 
\end{align*}

We also find that
$$B_2 =  \big\langle  e^{i\zeta}  W_\mu , -i \Delta^2 g +i \big(f(e^{i\zeta} W_\mu +e^{i\theta} W_\lambda +g) - f(e^{i\zeta} W_\mu) - f(e^{i\theta} W_\lambda)\big) \big\rangle. $$
Proceeding as previously, we get that
$$|B_2|  \lesssim  \|g\|_{\mathcal{E}}^2+\lambda^{\frac{N-4}{2}}+ \lambda^{\frac{N-4}{4}} \|g\|_{\mathcal{E}}\lesssim 
|t|^{-\frac{N-4}{N-12}}.$$

{\bf Third row.} 

Differentiating $\langle ie^{i\theta}\Lambda  W_\lambda ,g\rangle =0$ with respect to time, we have
\begin{align*}
0&= \frac{d}{dt} \langle ie^{i\theta}\Lambda W_\lambda ,g\rangle = - \theta^\prime \langle e^{i\theta} \Lambda W_\lambda ,g \rangle - \frac{\lambda^\prime}{\lambda} \langle ie^{i\theta} \Lambda  \Lambda W_\lambda ,g \rangle + \langle ie^{i\theta}\Lambda  W_\lambda ,\partial_t g \rangle \\
&= \zeta^\prime \langle i e^{i\theta}\Lambda  W_\mu ,- i e^{i\zeta} W_\mu \rangle  +\frac{\mu^\prime}{\mu} \langle  i e^{i\theta} \Lambda W_\lambda , e^{i\zeta} \Lambda W_\mu\rangle\\
&+\theta^\prime \big(\langle ie^{i\theta}\Lambda  W_\lambda ,-i e^{i\theta} W_\lambda \rangle - \langle e^{i\theta} \Lambda W_\lambda ,g\rangle\big) +\frac{\lambda^\prime}{\lambda}\big( \langle ie^{i\theta} \Lambda  W_\lambda , e^{i\theta} \Lambda W_\lambda \rangle -\langle ie^{i\theta} \Lambda \Lambda W_\lambda ,g \rangle \big) \\
&+ \big\langle i e^{i\theta} \Lambda  W_\lambda , -i \Delta^2 g +i \big(f(e^{i\zeta} W_\mu +e^{i\theta} W_\lambda +g) - f(e^{i\zeta} W_\mu) - f(e^{i\theta} W_\lambda)\big) \big\rangle .
\end{align*}

Therefore, we find
\begin{align*}
M_{31} &= \mu^{-4}\langle i e^{i\theta}  \Lambda W_\lambda , - i e^{i\zeta} W_\mu \rangle  = O\big(\lambda^{\frac{N-4}{2}}\big)
=O\big(|t|^{-\frac{N-4}{N-12}}\big),\\
M_{32}&= \mu^{-4}  \langle  i e^{i\theta} \Lambda W_\lambda , e^{i\zeta} \Lambda W_\mu\rangle  = O\big(\lambda^{\frac{N-4}{2}}\big)
=O\big(|t|^{-\frac{N-4}{N-12}}\big),\\
M_{33}&= \lambda^{-4}  \big(\langle ie^{i\theta}\Lambda  W_\lambda ,-i e^{i\theta} W_\lambda \rangle - \langle e^{i\theta} \Lambda W_\lambda ,g\rangle\big)  =2\|W\|_{L^2}^2 +O(\|g\|_{\mathcal{E}})\\
&= 2 \|W\|_{L^2}^2 +O\big(|t|^{-\frac{N-3}{2(N-12)}}\big),\\
\noalign{and}
M_{34}&= \lambda^{-4} \big( \langle ie^{i\theta} \Lambda  W_\lambda , e^{i\theta} \Lambda W_\lambda \rangle -\langle ie^{i\theta} \Lambda \Lambda W_\lambda ,g \rangle \big)= O(\|g\|_{\mathcal{E}}) =
O\big(|t|^{-\frac{N-3}{2(N-12)}}\big).
\end{align*}

As before, $B_3$ can be rewritten as 
$$B_3 =-  \langle  e^{i\theta}  \Lambda W_\lambda ,  f(e^{i\zeta} W_\mu +e^{i\theta} W_\lambda +g) - f(e^{i\zeta} W_\mu) - f(e^{i\theta} W_\lambda) -f^\prime (e^{i\theta} W_\lambda) g \rangle .$$
We set
$$K= - \langle  e^{i\theta}  \Lambda W_\lambda ,  f(e^{i\zeta} W_\mu +e^{i\theta} W_\lambda +g) - f(e^{i\zeta} W_\mu +e^{i\theta} W_\lambda) -f^\prime (e^{i\theta} W_\lambda +e^{i\zeta} W_\mu) g \rangle .$$
 First we will estimate $B_3 - K$ by writing
\begin{align*}
B_3 - K&=- \big\langle  e^{i\theta}  \Lambda W_\lambda , \big( f^\prime (e^{i\theta} W_\lambda +e^{i\zeta} W_\mu) - f^\prime (e^{i\theta} W_\lambda )\big) g \big\rangle  \\
& - \langle  e^{i\theta}  \Lambda W_\lambda ,  f(e^{i\zeta} W_\mu +e^{i\theta} W_\lambda ) - f(e^{i\zeta} W_\mu) -f(e^{i\theta} W_\lambda)-f^\prime (e^{i\theta} W_\lambda) e^{i\zeta} W_\mu \rangle  \\
&- \langle e^{i\theta} \Lambda W_\lambda , f^\prime (e^{i\theta} W_\lambda) e^{i\zeta} W_\mu \rangle + C_2 \theta \lambda^{\frac{N-4}{2}} -C_2 \theta \lambda^{\frac{N-4}{2}} \\
&= B_{31}+B_{32}+ B_{33}- C_2 \theta \lambda^{\frac{N-4}{2}},
\end{align*}
where $C_2=- \frac{N+4}{N-4} \lambda^{-\frac{N-4}{2}} \int_{\R^N} W_\lambda^{\frac{8}{N-4}} \Lambda W_\lambda dx= \frac{N+4}{N-4} \int_{\R^N} W^{\frac{8}{N-4}} \Lambda W dx $. We will show that
$$ |B_{31}|= \big| \big\langle  e^{i\theta}  \Lambda W_\lambda , \big( f^\prime (e^{i\theta} W_\lambda +e^{i\zeta} W_\mu) - f^\prime (e^{i\theta} W_\lambda )\big) g \big\rangle \big| \lesssim \lambda^{\frac{N}{4} } \|g\|_{\mathcal{E}},$$
$$|B_{32}|=\big|\langle  e^{i\theta}  \Lambda W_\lambda ,  f(e^{i\zeta} W_\mu +e^{i\theta} W_\lambda ) - f(e^{i\zeta} W_\mu) -f(e^{i\theta} W_\lambda) - f^\prime (e^{i\theta} W_\lambda) e^{i\zeta} W_\mu\rangle \big| \lesssim \lambda^{\frac{N}{2}},$$
and
\begin{align*}
|B_{33}|&=|\langle e^{i\theta} \Lambda W_\lambda , f^\prime (e^{i\theta} W_\lambda) (e^{i\zeta} W_\mu)\rangle + C_2 \theta \lambda^{\frac{N-4}{2}}|\\
&\lesssim   ( |\theta|^3 + |\zeta +\pi/2| ) \lambda^{\frac{N-4}{2}} +  \lambda^{\frac{N-2}{2}} +|1-\mu| \lambda^{\frac{N-4}{2}}.
\end{align*}
So, by above estimates and noticing that $|K|\lesssim  \|g\|^2_{\mathcal{E}} $ (cf. \eqref{b11}), we find
\begin{align*}
|B_3|&\lesssim |B_3 -K |+|K|\\
&\lesssim \lambda^{\frac{N}{4} } \|g\|_{\mathcal{E}} +\lambda^{\frac{N}{2}}+  ( |\theta|^3 + |\zeta +\pi /2| ) \lambda^{\frac{N-4}{2}} +  \lambda^{\frac{N-2}{2}} +|1-\mu| \lambda^{\frac{N-4}{2}} + |\theta| \lambda^{\frac{N-4}{2}}+ \|g\|^2_{\mathcal{E}} \\
&\lesssim \lambda^{\frac{N-2}{2}} + |\theta|\lambda^{\frac{N-4}{2}}+ \|g\|^2_{\mathcal{E}}
\lesssim 
|t|^{-\frac{N-3}{N-12}}.
\end{align*}
So we are left with estimating the $B_{3i}$, $i=1,2,3$. We begin with $B_{31}$. We will split the integral into two parts depending whether 
 $|x|\leq \lambda^\gamma$, with $\gamma =\frac{N-8}{2(N-4)}$, or not. First, notice that
$$\big|  e^{i\theta}  \Lambda W_\lambda \big|\big|\big(f^\prime (e^{i\theta} W_\lambda +e^{i\zeta} W_\mu) - f^\prime (e^{i\theta} W_\lambda) \big)g\big|\lesssim W_\lambda^{\frac{8-N}{N-4}} W_\mu |g|\lesssim W_\lambda^{\frac{8}{N-4}} |g|  .$$
So, using H\"older's inequality, we find
\begin{align*}
 &\big| \big\langle 1_{\{|x|\leq \lambda^\gamma \}}  e^{i\theta}  \Lambda W_\lambda , \big( f^\prime (e^{i\theta} W_\lambda +e^{i\zeta} W_\mu) - f^\prime (e^{i\theta} W_\lambda )\big) g \big\rangle \big| \\ 
&\lesssim \|W_\lambda^{\frac{8}{N-4}}\|_{L^{\frac{2N}{N+4}} (|x|\leq \lambda^\gamma)} \|g\|_{\mathcal{E}} \\
&\lesssim \lambda^{\frac{N-4}{2}} \left(\int_0^{\lambda^{\gamma -1}}  r^{-\frac{16N}{N+4}} r^{N-1}dr\right)^{\frac{N+4}{2N}} \|g\|_{\mathcal{E}}\\
&\lesssim \lambda^{(\gamma -1) \frac{N- 12}{2} +\frac{N-4}{2}} \|g\|_{\mathcal{E}}\\
&\lesssim \lambda^{\frac{N}{4}} \|g\|_{\mathcal{E}}.
\end{align*}

On the other hand, when $|x|\geq \lambda^\gamma$, we have
\begin{align*}
\|\Lambda W_\lambda \|_{L^{\frac{2N}{N-4}}(|x|\geq \lambda^\gamma) } &= \|W\|_{L^{\frac{2N}{N-4}} (|x| \geq \lambda^{\gamma -1})} \\
&\lesssim \Big(\int_{\lambda^{\gamma -1} }^\infty r^{-2N} r^{N-1} dr \Big)^{\frac{N-4}{2N}}\\
&\lesssim \lambda^{(1-\gamma)\frac{N-4}{2}}=\lambda^{\frac{N}{4}}.
\end{align*}
Using H\" older's inequality as well as the fact that $ \|\Lambda W_\mu\|_{L^{\frac{N}{4}}(|x|\geq  \lambda^\gamma)} \lesssim 1$, we get that
$$|B_{31}|\lesssim  \lambda^{\frac{N}{4}} \|g\|_{\mathcal{E}}.$$

Next, let us estimate $B_{32}$. Notice that, if
 $|x|\geq \sqrt{\lambda}$, we have
$$|  f(e^{i\zeta} W_\mu +e^{i\theta} W_\lambda ) - f(e^{i\zeta} W_\mu) -f(e^{i\theta} W_\lambda) - f^\prime (e^{i\theta} W_\lambda) (e^{i\zeta} W_\mu)|\lesssim W_\lambda^{\frac{8}{N-4}} W_\mu.$$
Since $|\Lambda W|\lesssim W$, we have
$$\|W_\lambda^{\frac{N+4}{N-4}}\|_{L^1 (|x|\geq \sqrt{\lambda})} \lesssim \lambda^{\frac{N-4}{2}} \int_{\lambda^{-\frac{1}{2}}}^\infty r^{-(N+4) } r^{N-1}dr \lesssim \lambda^{\frac{N}{2}}. $$
So the estimate follows from the fact that $W_\mu \lesssim 1$. Now, if $|x|\leq \sqrt{\lambda}$, notice that $W_\mu \lesssim W_\lambda$, so
$$|  f(e^{i\zeta} W_\mu +e^{i\theta} W_\lambda ) - f(e^{i\zeta} W_\mu) -f(e^{i\theta} W_\lambda) - f^\prime (e^{i\theta} W_\lambda) (e^{i\zeta} W_\mu)|\leq W_\mu^{\frac{N+4}{N-4}}\lesssim 1.$$
Then the estimate of $B_{32}$ follows from
$$\|W_\lambda\|_{L^1 (|x|\leq \sqrt{\lambda})} \lesssim \lambda^{\frac{N+4}{2}} \int_0^{\frac{1}{\sqrt{\lambda}}} r^{4-N} r^{N-1} dr \leq \lambda^{\frac{N}{2}}.$$
Finally, we estimate $B_{33}$.
Recall that, by definition,
$$ f^\prime (e^{i\theta} W_\lambda) e^{i\zeta} W_\mu = W_\mu W_\lambda^{\frac{8}{N-4}} \left(e^{i\zeta} + \tfrac{8}{N-4} e^{i\theta} \mathcal{R} (e^{i(\zeta -\theta)})\right).$$
This implies that
$$\langle e^{i\theta} \Lambda W_\lambda , f^\prime (e^{i\theta} W_\lambda) e^{i\zeta} W_\mu\rangle = \frac{N+4}{N-4} \mathcal{R} \left(e^{i(\zeta -\theta)}\right) \int_{\R^N} W_\mu W_\lambda^{\frac{8}{N-4}} \Lambda W_\lambda dx.$$
 Since $ \int_{\R^N} W_\mu W_\lambda^{\frac{8}{N-4}} \Lambda W_\lambda dx \lesssim \lambda^{\frac{N-4}{2}}$ and $|\mathcal{R}(e^{i(\zeta -\theta)}) +\theta|\lesssim |\theta|^3 + |\zeta +\pi /2|$, we get
\begin{align*}
&\left|  \frac{N+4}{N-4} \mathcal{R} \left(e^{i(\zeta -\theta)}\right) \int_{\R^N} W_\mu W_\lambda^{\frac{8}{N-4}} \Lambda W_\lambda dx + \frac{N+4}{N-4} \theta \int _{\R^N}W_\mu  W_\lambda^{\frac{8}{N-4}} \Lambda W_\lambda dx\right|\\
& \lesssim ( |\theta|^3 + |\zeta +\pi/2| ) \lambda^{\frac{N-4}{2}}. 
\end{align*}
At this point, the estimate of $B_{33}$ follows from the estimate
\begin{equation}\label{3-26}
\left|\int_{\R^N} W_\lambda^{\frac{8}{N-4}} \Lambda W_\lambda dx - \int_{\R^N} W_\mu W_\lambda^{\frac{8}{N-4}} \Lambda W_\lambda dx\right| \lesssim \lambda^{\frac{N-2}{2}} +|1-\mu| \lambda^{\frac{N-4}{2}}.
\end{equation}
One can verify the estimate above by rescaling if $|x|\geq \sqrt{\lambda}$. When $|x|\leq \sqrt{\lambda}$, we have $|W_\mu - \mu^{-(N-4)/2}|\lesssim |x|^2 \lesssim \lambda$ and $|\mu^{-(N-4)/2}-1 |\leq |1-\mu|$ and \eqref{3-26} follows.

{\bf Fourth row.} 

Finally we deal with the fourth row. In this case, we differentiate $\langle -e^{i\theta}  W_\lambda ,g\rangle =0$ with respect to time to find
\begin{align*}
0&= \frac{d}{dt} \langle -e^{i\theta} W_\lambda ,g\rangle = - \theta^\prime \langle ie^{i\theta}  W_\lambda ,g \rangle + \frac{\lambda^\prime}{\lambda} \langle e^{i\theta}   \Lambda W_\lambda ,g \rangle - \langle e^{i\theta}  W_\lambda ,\partial_t g \rangle \\
&= \zeta^\prime \langle  e^{i\theta}  W_\lambda , i e^{i\zeta} W_\mu \rangle  -\frac{\mu^\prime}{\mu} \langle   e^{i\theta}  W_\lambda , e^{i\zeta} \Lambda W_\mu\rangle\\
&+\theta^\prime \big(\langle e^{i\theta}  W_\lambda ,i e^{i\theta} W_\lambda \rangle - \langle i e^{i\theta}  W_\lambda ,g\rangle\big) +\frac{\lambda^\prime}{\lambda}\big( \langle - e^{i\theta}   W_\lambda , e^{i\theta} \Lambda W_\lambda \rangle +\langle e^{i\theta}  \Lambda W_\lambda ,g \rangle \big) \\
&- \big\langle  e^{i\theta} W_\lambda , -i \Delta^2 g +i \big(f(e^{i\zeta} W_\mu +e^{i\theta} W_\lambda +g) - f(e^{i\zeta} W_\mu) - f(e^{i\theta} W_\lambda)\big) \big\rangle .
\end{align*}
Therefore, we get
\begin{align*}
M_{41} &= \mu^{-4}\langle i e^{i\theta}  W_\lambda ,  i e^{i\zeta} W_\mu \rangle  = O\big(\lambda^{\frac{N-4}{2}}\big)=O\big(|t|^{-\frac{N-4}{N-12}}\big),\\
M_{42}&= \mu^{-4}  \langle   e^{i\theta}  W_\lambda , e^{i\zeta} \Lambda W_\mu\rangle  = O\big(\lambda^{\frac{N-4}{2}}\big)
=O\big(|t|^{-\frac{N-4}{N-12}}\big),\\
M_{43}&= \lambda^{-4}  \big(\langle e^{i\theta} W_\lambda ,i e^{i\theta} W_\lambda \rangle - \langle i e^{i\theta}  W_\lambda ,g\rangle\big)  = O(\|g\|_{\mathcal{E}})
 =
O\big(|t|^{-\frac{N-3}{2(N-12)}}\big),\\
\noalign{and}
M_{44}&= \lambda^{-4} \big( \langle -e^{i\theta}   W_\lambda , e^{i\theta} \Lambda W_\lambda \rangle +\langle e^{i\theta} \Lambda \Lambda W_\lambda ,g \rangle \big)=2 \|W\|_{L^2}^2 +O(\|g\|_{\mathcal{E}})\\
&= 2 \|W\|_{L^2}^2 +O\big(|t|^{-\frac{N-3}{2(N-12)}}\big).
\end{align*}

We also have that
$$B_4 =  \big\langle  e^{i\theta}   W_\lambda , i\big( f(e^{i\zeta} W_\mu +e^{i\theta} W_\lambda +g) - f(e^{i\zeta} W_\mu) - f(e^{i\theta} W_\lambda) -f^\prime (e^{i\theta} W_\lambda) g \big)\big\rangle .$$
We rewrite $B_4$ as follows
\begin{align*}
B_4 &=\big\langle  e^{i\theta}   W_\lambda , i\big( f(e^{i\zeta} W_\mu +e^{i\theta} W_\lambda +g) - f(e^{i\zeta} W_\mu +e^{i\theta}W_\lambda)  -f^\prime (e^{i\theta} W_\lambda + e^{i\zeta} W_\mu) g \big)\big\rangle\\
&+ \big\langle  e^{i\theta}   W_\lambda ,  i\big( f^\prime (e^{i\zeta} W_\mu +e^{i\theta} W_\lambda )g - f^\prime (e^{i\theta} W_\lambda)g \big)\big\rangle \\
&+ \big\langle  e^{i\theta}   W_\lambda , i\big( f(e^{i\zeta} W_\mu +e^{i\theta} W_\lambda ) - f(e^{i\zeta} W_\mu)-f (e^{i\theta} W_\lambda) -f^\prime (e^{i\theta} W_\lambda) e^{i\zeta} W_\mu \big)\big\rangle\\
&+ \langle e^{i\theta } W_\lambda , if^\prime  (e^{i\theta} W_\lambda) e^{i\zeta} W_\mu \rangle .
\end{align*}
Since $|\zeta +\pi/2|\approx 0$ and $\theta \approx 0$, we see that $|e^{i\zeta} W_\mu +e^{i\theta} W_\lambda |\gtrsim W_\lambda$. So, using Taylor's expansion \eqref{fe4}, we find
$$| f(e^{i\zeta} W_\mu +e^{i\theta} W_\lambda +g) - f(e^{i\zeta} W_\mu +e^{i\theta} W_\lambda) -f^\prime (e^{i\theta} W_\lambda +e^{i\zeta} W_\mu) g| \lesssim W_\lambda^{\frac{12-N}{N-4}} |g|^2. $$
Using that $|\Lambda W|\lesssim W$ and H\"older's inequality, we deduce that
$$\big|   \big\langle  e^{i\theta}   W_\lambda , i\big( f(e^{i\zeta} W_\mu +e^{i\theta} W_\lambda +g) - f(e^{i\zeta} W_\mu +e^{i\theta} W_\lambda) -f^\prime (e^{i\theta} W_\lambda +e^{i\zeta} W_\mu) g \big)\big\rangle\big| \lesssim \|g\|_{\mathcal{E}}^2 .$$
Proceeding as for $B_{31}$, one can show that
$$\big|  \big\langle  e^{i\theta}   W_\lambda , i\big( f^\prime (e^{i\zeta} W_\mu +e^{i\theta} W_\lambda )g - f^\prime (e^{i\theta} W_\lambda)g\big) \big\rangle \big|\lesssim \lambda^{\frac{N}{4}} \|g\|_{\mathcal{E}}.$$
We can also apply the proof of the estimate of $M_{32}$ to deduce that
$$ \big|  \big\langle  e^{i\theta}   W_\lambda , i\big( f(e^{i\zeta} W_\mu +e^{i\theta} W_\lambda ) - f(e^{i\zeta} W_\mu)-f (e^{i\theta} W_\lambda) -f^\prime (e^{i\theta} W_\lambda)( e^{i\zeta} W_\mu) \big)\big\rangle\big| \lesssim \lambda^{\frac{N}{2}}.$$
We are left with estimating
$$\langle e^{i\theta } W_\lambda , if^\prime  (e^{i\theta} W_\lambda) e^{i\zeta} W_\mu \rangle = \mathcal{R} (ie^{i(\zeta -\theta)}) \int_{\R^N} W_\mu W_\lambda^{\frac{N+4}{N-4}}dx.$$
Notice that
$$|\mathcal{R} (ie^{i (\zeta -\theta)}) -1|\leq |\mathcal{R}(i e^{i (\zeta - \theta)} -e^{-i\theta} )| +|\mathcal{R} (e^{-i\theta}) -1| \leq  |\theta|^2 + |\zeta +\pi/2|.$$
Since $\big|\int_{\R^N} W_\mu W_\lambda^{\frac{N+4}{N-4}} dx\big|\lesssim \lambda^{\frac{N-4}{2}}$, we deduce that
$$ \left| \big(\mathcal{R} (ie^{i (\zeta -\theta)})  -1\big) \int_{\R^N} W_\mu W_\lambda^{\frac{N+4}{N-4}}dx \right|\lesssim \lambda^{\frac{N-4}{2}} (|\theta|^2 + |\zeta+\pi/2|). $$
Proceeding as for the last estimate in the proof of the estimate of $M_{33}$, one can show that
$$\left|\int_{\R^N} W_\lambda^{\frac{N+4}{N-4}} -\int_{\R^N} W_\mu W_\lambda^{\frac{N+4}{N-4}} dx\right| \lesssim \lambda^{\frac{N-2}{2}} + |\mu -1| \lambda^{\frac{N-4}{2}}.$$
Combining all the previous estimates and since 
$$\int_{\R^N} W_\lambda^{\frac{N+4}{N-4}} dx =C_1 \lambda^{\frac{N-4}{2}},$$
we finally deduce that
$$|B_4 - C_1 \lambda^{\frac{N-4}{2}}|\lesssim \|g\|^2_{\mathcal{E}} + \lambda^{\frac{N}{4}} \|g\|_{\mathcal{E}}+ \lambda^{\frac{N}{2}}+ |\theta|^2 \lambda^{\frac{N-4}{2}}+\lambda^{\frac{N-2}{2}} + |\mu -1| \lambda^{\frac{N-4}{2}}. $$
This implies that
$$|B_4|\lesssim \lambda^{\frac{N-4}{2}}+\|g\|^2_{\mathcal{E}}
\lesssim  |t|^{-\frac{N-4}{N-12}}.
$$

Let us denote $\beta=|t|^{-\frac{N-3}{2(N-12)}}$ and $\beta_1=|t|^{-\frac{N-4}{N-12}}$.
Then, we have
$$M=\begin{pmatrix}2\|W\|_{L^2}^2 +O(\beta) & O(\beta) & O(1) & O(1)\\ O(\beta) & 2\|W\|_{L^2}^2 +O(\beta) & O(1) & O(1)\\ O(\beta_1) &O(\beta_1) & 2\|W\|_{L^2}^2+O(\beta) & O(\beta) \\ O(\beta_1) & O(\beta_1) &O(\beta) &   2\|W\|_{L^2}^2+O(\beta)  \end{pmatrix}.$$
It follows that $M^{-1}$ exists and is of the same form as $M$ with $\|W\|_{L^2}^2$ substituted by $2^{-1}\|W\|_{L^2}^{-2}$. Since 
$$|B_i| \lesssim \lambda^{\frac{N-4}{2}}
\lesssim  |t|^{-\frac{N-4}{N-12}},$$
we deduce that 
\begin{equation}\label{zetamu}
|\zeta^\prime| +|\mu^\prime|\lesssim \lambda^{\frac{N-4}{2}}\lesssim |t|^{-\frac{N-4}{N-12}}.
\end{equation}
 We also deduce that
$$\left|\lambda^4 \theta^\prime - \left(\frac{1}{2\|W\|_{L^2}^2} +O(\beta)\right) B_3 \right|\lesssim \beta_1 \lambda^{\frac{N-4}{2}},$$
and
$$\left|\lambda^3 \lambda^\prime - \left(\frac{1}{2\|W\|_{L^2}^2} +O(\beta)\right) B_4 \right|\lesssim \beta_1 \lambda^{\frac{N-4}{2}}.$$
Since 
$$\big|B_3 - K+ C_2 \theta \lambda^{\frac{N-4}{2}}\big| \lesssim  \lambda^{\frac{N}{4} } \|g\|_{\mathcal{E}} +\lambda^{\frac{N}{2}}+  ( |\theta|^3 + |\zeta +\pi /2| ) \lambda^{\frac{N-4}{2}} +  \lambda^{\frac{N-2}{2}} +|1-\mu| \lambda^{\frac{N-4}{2}},$$
 and
$$\big|B_4- C_1 \lambda^{\frac{N-4}{2}}\big|\leq \|g\|^2_{\mathcal{E}} + \lambda^{\frac{N}{4}} \|g\|_{\mathcal{E}}+ \lambda^{\frac{N}{2}}+ |\theta|^2 \lambda^{\frac{N-4}{2}}+\lambda^{\frac{N-2}{2}} + |\mu -1| \lambda^{\frac{N-4}{2}},$$ 
we get that
\begin{align*}
\Big|\theta^\prime -&\frac{1}{2\lambda^4 \|W\|_{L^2}^2} \big(K- C_2 \theta \lambda^{\frac{N-4}{2}}\big) \Big| \lesssim \beta_1 \lambda^{\frac{N-4}{2} -4}\\
&+  \frac{1}{\lambda^4}\big(\lambda^{\frac{N}{4} } \|g\|_{\mathcal{E}} +\lambda^{\frac{N}{2}}+  ( |\theta|^3 + |\zeta +\pi /2| ) \lambda^{\frac{N-4}{2}} +  \lambda^{\frac{N-2}{2}} +|1-\mu| \lambda^{\frac{N-4}{2}} \big)  \\
&\lesssim |t|^{-\frac{N-10}{N-12}}
\end{align*}
and
\begin{align*}
\Big|\lambda^\prime -& \frac{C_1\lambda^{\frac{N-4}{2}}}{2\lambda^3 \|W\|_{L^2}^2} \Big| \lesssim \beta_1 \lambda^{\frac{N-4}{2}-3}\\
&+\frac{1}{\lambda^3} \big(\|g\|^2_{\mathcal{E}} + \lambda^{\frac{N}{4}} \|g\|_{\mathcal{E}}+ \lambda^{\frac{N}{2}}+ |\theta|^2 \lambda^{\frac{N-4}{2}}+\lambda^{\frac{N-2}{2}} + |\mu -1| \lambda^{\frac{N-4}{2}}\big)\\
&\lesssim |t|^{-\frac{N-9}{N-12}}.
\end{align*}

\end{proof}

We conclude this section by analysing the stable and unstable directions of the linearized flow (see \eqref{eigenvprel1} and \eqref{eigenvprel2}). To do so, we define \[
a_1^\pm (t)= \langle \alpha_{\zeta (t) , \mu (t)}^\pm , g(t) \rangle
\] 
and 
\[
a_2^\pm (t)= \langle \alpha_{\theta (t) , \lambda (t)}^\pm , g(t) \rangle,
\]  
with $a_i^-$ corresponding to the stable direction and $a_i^+$ to the unstable one. We will show that $\frac{d}{dt}a_i^\pm (t)$ behaves roughly as $\mp \frac{\nu}{\lambda_i (t)^4} a_i^\pm (t) $, where $\lambda_1 (t)=\mu (t)$ and $\lambda_2 (t)=\lambda (t)$ and $\nu$ is defined as in \eqref{eigenvprel1} and \eqref{eigenvprel2}.

\begin{lem}
\label{lemmodeai}
Under the same assumptions as in Proposition \ref{lem3.1}, we have, for $t\in [T,T_1]$,
\begin{equation}
\label{stcoe1}
\left|\frac{d}{dt}a_1^+ (t) - \frac{\nu}{\mu (t)^4} a_1^+ (t)\right| \leq \frac{c}{\mu(t)^4} |t|^{-\frac{N}{2(N-12)}},
\end{equation}
\begin{equation}
\label{stcoe2}
\left|\frac{d}{dt}a_1^- (t) + \frac{\nu}{\mu (t)^4} a_1^- (t)\right| \leq \frac{c}{\mu(t)^4} |t|^{-\frac{N}{2(N-12)}},
\end{equation}
\begin{equation}
\label{stcoe3}
\left|\frac{d}{dt}a_2^+ (t) - \frac{\nu}{\lambda (t)^4} a_2^+ (t)\right| \leq \frac{c}{\lambda (t)^4} |t|^{-\frac{N}{2(N-12)}},
\end{equation}
\begin{equation}
\label{stcoe4}
\left|\frac{d}{dt}a_2^- (t) + \frac{\nu}{\lambda (t)^4} a_2^- (t)\right| \leq \frac{c}{\lambda (t)^4} |t|^{-\frac{N}{2(N-12)}},
\end{equation}
with $c\to 0$ as $|T_0|\to \infty$.

\end{lem}

\begin{proof}
Using the definition of $\alpha_{\zeta, \mu}^+ $, we have 
\begin{align*}
\dfrac{d}{dt} a_1^+ &= -\frac{4\mu^\prime}{\mu} \Big\langle \frac{e^{i\zeta}}{\mu^4} \big(Y^{(2)}_\mu +i  Y_\mu^{(1)}\big),g \Big\rangle
-\frac{\mu^\prime}{\mu} \Big\langle \frac{e^{i\zeta}}{\mu^4} \big(\Lambda Y^{(2)}_\mu +i \Lambda Y_\mu^{(1)}\big),g \Big\rangle\\
&\quad +\zeta^\prime \Big\langle \frac{e^{i\zeta}}{\mu^4} \big(iY_\mu^{(2)} - Y_\mu^{(1)} \big) ,g  \Big\rangle + \langle \alpha_{\zeta ,\mu}^+ , \partial_t g \rangle \\
&=:-\frac{\mu^\prime}{\mu} \Big\langle \frac{e^{i\zeta}}{\mu^4} \big(\Lambda_{-2} Y^{(2)}_\mu +i \Lambda_{-2} Y_\mu^{(1)}\big),g \Big\rangle +\zeta^\prime \Big\langle \frac{e^{i\zeta}}{\mu^4} \big(iY_\mu^{(2)} - Y_\mu^{(1)} \big) ,g  \Big\rangle + \langle \alpha_{\zeta ,\mu}^+ , \partial_t g \rangle.
\end{align*}
Since $\|\mu^{-4} Y^{(i)}_\mu\|_{L^{\frac{2N}{N+4}}} \lesssim 1$, for $i=1,2$, the first two terms above can be estimated as
\begin{align}\label{first3}
&\Big| -\frac{\mu^\prime}{\mu} \Big\langle \frac{e^{i\zeta}}{\mu^4} \big(\Lambda_{-2} Y^{(2)}_\mu +i \Lambda_{-2} Y_\mu^{(1)}\big),g \Big\rangle +\zeta^\prime \Big\langle \frac{e^{i\zeta}}{\mu^4} \big(iY_\mu^{(2)} - Y_\mu^{(1)} \big) ,g  \Big\rangle \Big|\nonumber\\
&\qquad \lesssim (|\mu^\prime| + |\zeta^\prime|)  \|g\|_{\mathcal{E}} \lesssim |t|^{-\frac{3N-11}{2(N-12)}}
\end{align}
by using \eqref{zetamu} and \eqref{mode5}. 
To estimate the third term, we recall that
\begin{align}
\label{eqdeg}
\partial_t g &=-i \Delta^2 g +i\big(f(e^{i\zeta} W_\mu +e^{i\theta} W_\lambda +g) - f(e^{i\zeta} W_\mu ) - f(e^{i\theta }W_\lambda)\big)\\
&-  \zeta^\prime i e^{i\zeta} W_\mu + \frac{\mu^\prime}{\mu} e^{i\zeta}\Lambda W_\mu -\theta^\prime i e^{i\theta} W_\lambda + \frac{\lambda^\prime}{\lambda}e^{i\theta} \Lambda W_\lambda  .\nonumber
\end{align}
Next we observe that
\begin{equation}\label{nexttwo}
\big\langle \alpha^+_{\zeta ,\mu} , -\zeta^\prime i e^{i\zeta} W_\mu  \big\rangle =\Big\langle \alpha^+_{\zeta ,\mu} ,  \frac{\mu^\prime}{\mu} e^{i\zeta}\Lambda W_\mu \Big\rangle =0.
\end{equation}
Notice that $\|W_\lambda \|_{\dot{H}^s} \lesssim \lambda^{2-s}$ and $\|\alpha_{\zeta ,\mu}^+\|_{\dot{H}^2}\lesssim 1$. So, this together with \eqref{thetalambdaprime} yield to
\begin{equation}\label{thirdest}
|\langle \alpha^+_{\zeta ,\mu} ,\theta^\prime i e^{i\theta} W_\lambda  \rangle  | \lesssim |\theta^\prime| \|\alpha_{\zeta ,\mu}^+\|_{\dot{H}^2}\|W_\lambda \|_{\dot{H}^{-2}} \lesssim |\theta^\prime| \lambda^4
\lesssim |t|^{-\frac{^{N-4}}{N-12}}
\end{equation}
and
\begin{equation}\label{fourthest}
\Big| \Big\langle \alpha^+_{\zeta ,\mu} ,\frac{\lambda^\prime}{\lambda} \Lambda W_\lambda  \Big\rangle  \Big| \lesssim \Big|\frac{\lambda^\prime}{\lambda}\Big| \|\alpha_{\zeta ,\mu}^+\|_{\dot{H}^2}\|\Lambda W_\lambda \|_{\dot{H}^{-2}} \lesssim \Big|\frac{\lambda^\prime}{\lambda}\Big| \lambda^4
\lesssim |t|^{-\frac{N-4}{N-12}} . 
\end{equation}
Using the definition of $Z_{\zeta ,\mu}$, we remark that
\begin{align*}
&-i \Delta^2 g  +i\big(f(e^{i\zeta} W_\mu +e^{i\theta} W_\lambda +g) - f(e^{i\zeta} W_\mu ) - f(e^{i\theta }W_\lambda)\big)\\
&=Z_{\zeta ,\mu}g +i\big(f(e^{i\zeta} W_\mu +e^{i\theta} W_\lambda +g) - f(e^{i\zeta} W_\mu ) - f(e^{i\theta }W_\lambda) - f^\prime (e^{i\zeta} W_\mu) g\big).
\end{align*}
We recall that $\langle \alpha^+_{\zeta ,\mu} , Z_{\zeta ,\mu}g \rangle = \frac{\nu}{\mu^4} a_1^+$. As in the proof of the estimate of $|B_1|$, we obtain 
\begin{align}\label{fiftest}
&\big|\big\langle \alpha^+_{\zeta ,\mu} , i\big(f(e^{i\zeta} W_\mu +e^{i\theta} W_\lambda +g) - f(e^{i\zeta} W_\mu ) - f(e^{i\theta }W_\lambda) - f^\prime (e^{i\zeta} W_\mu) g\big) \big\rangle\big|\nonumber\\
&\qquad \lesssim \lambda^{\frac{N-4}{2}}
\lesssim |t|^{-\frac{N-4}{N-12}}.
\end{align}
So combining \eqref{first3}--\eqref{fiftest}, we find that
$$\Big|\frac{d}{dt}a_1^+ (t) - \frac{\nu}{\mu (t)^4} a_1^+ (t)\Big| \lesssim  
|t|^{-\frac{3N-11}{2(N-12)}}\ll \frac{1}{\mu(t)^4}|t|^{-\frac{N}{2(N-12)}}.
$$
This establishes \eqref{stcoe1}.

Next, let us turn to the proof of \eqref{stcoe3}. By definition of $\alpha_{\theta,\lambda}^+$, we find 
$$\frac{d}{dt} a_2^+ = -\frac{\lambda^\prime}{\lambda} \Big\langle \frac{e^{i\theta}}{\lambda^4} \big(\Lambda_{-2} Y^{(2)}_\lambda +i \Lambda_{-2} Y_\lambda^{(1)}\big),g \Big\rangle +\theta^\prime \Big\langle \frac{e^{i\theta}}{\lambda^4} \big(iY_\lambda^{(2)} - Y_\lambda^{(1)} \big) ,g  \Big\rangle + \langle \alpha_{\theta ,\lambda}^+ , \partial_t g \rangle .$$
Using that $\|\lambda^{-4}Y_\lambda^{(i)}\|_{L^{\frac{2N}{N+4}}}\lesssim 1$, we obtain
\begin{align}\label{2ndfirst3}
&\Big| -\frac{\lambda^\prime}{\lambda} \Big\langle \frac{e^{i\theta}}{\lambda^4} \big(\Lambda_{-2} Y^{(2)}_\lambda +i \Lambda_{-2} Y_\lambda^{(1)}\big),g \Big\rangle +\theta^\prime \Big\langle \frac{e^{i\theta}}{\lambda^4} \big(iY_\lambda^{(2)} - Y_\lambda^{(1)} \big) ,g  \Big\rangle \Big|\nonumber\\ 
&\qquad\lesssim  \Big(\Big|\frac{\lambda^\prime}{\lambda}\Big| + |\theta^\prime|\Big) \|g\|_{\mathcal{E}}\lesssim  |t|^{-\frac{3N-27}{2(N-12)}}.
\end{align}

So we are left with estimating the term $ \langle \alpha_{\theta ,\lambda}^+ , \partial_t g \rangle $. As previously, we will expand $\partial_t g$ using \eqref{eqdeg}. We get rid of two terms since
$$\langle \alpha^+_{\theta , \lambda} , - \theta^\prime i e^{i\theta}W_\lambda \rangle= \Big\langle \alpha_{\theta ,\lambda}^+ , \frac{\lambda^\prime}{\lambda} e^{i\theta} \Lambda W_\lambda \Big\rangle=0 .$$
We have
$$|\langle \alpha^+_{\theta ,\lambda} , - \zeta^\prime i e^{i\zeta} W_\mu \rangle |\leq \| \zeta^\prime i e^{i\zeta} W_\mu  \|_{L^\infty} \|\alpha_{\theta , \lambda}^+\|_{L^1}\lesssim |\zeta^\prime| \lambda^{\frac{N-4}{2}}. $$
We can estimate $\langle \alpha^+_{\theta ,\lambda} , \frac{\mu^\prime}{\mu}  e^{i\zeta} W_\mu \rangle $ in the same way.
Recall that $\langle \alpha_{\theta ,\lambda}^+ , Z_{\theta ,\lambda}g\rangle = \dfrac{\nu}{\lambda^4} a_2^+$. As for the estimate \eqref{fiftest}, our aim is to estimate
$$\lambda^4\big| \big\langle \alpha^+_{\theta ,\lambda} , i\big(f(e^{i\zeta} W_\mu +e^{i\theta} W_\lambda +g) - f(e^{i\zeta} W_\mu ) - f(e^{i\theta }W_\lambda) - f^\prime (e^{i\zeta} W_\mu) g\big) \big\rangle\big| .$$
Proceeding as in the estimate of $B_{3}$, we obtain
$$\lambda^4 \big|\big\langle \alpha_{\theta ,\lambda}^+ , i \big(f^\prime (e^{i\zeta} W_\mu +e^{i\theta} W_\lambda )- f^\prime (e^{i\theta } W_\lambda)\big)g\big\rangle\big|\lesssim \lambda^{\frac{N}{4}} \|g\|_{\mathcal{E}}.$$
Using the same kind of estimates as for $|B_{4}|$, we get
$$\lambda^4 \big|\big\langle \alpha_{\theta ,\lambda}^+ , i \big(f(e^{i\zeta} W_\mu +e^{i\theta} W_\lambda +g )-f(e^{i\zeta }W_\mu +e^{i\theta} W_\lambda)- f^\prime (e^{i\theta } W_\lambda +e^{i\zeta}W_\mu)g \big)\big\rangle\big|\lesssim \|g\|^2_{\mathcal{E}}.$$
Using a Taylor's expansion (see \eqref{fe2}), we have
$$\|f(e^{i\zeta} W_\mu +e^{i\theta} W_\lambda  )-f (e^{i\zeta} W_\mu) - f (e^{i\theta} W_\lambda )\|_{L^\infty} \lesssim \|W_\lambda^{\frac{8}{N-4}} W_\mu\|_{L^\infty}\lesssim \dfrac{1}{\lambda^4}. $$
Since $\| \alpha_{\theta ,\lambda}^+\|_{L^1} \lesssim \lambda^{\frac{N-4}{2}}$, we deduce that
$$\lambda^4  \big|\big\langle \alpha_{\theta ,\lambda}^+ , i \big(f (e^{i\zeta} W_\mu +e^{i\theta} W_\lambda  )-f (e^{i\zeta} W_\mu) - f (e^{i\theta} W_\lambda ) \big) \big\rangle \big| \lesssim  \lambda^{\frac{N-4}{2}}.$$
This yields to
\begin{align*}
\Big|\frac{d}{dt}a_2^+ (t) - \frac{\nu}{\lambda (t)^4} a_2^+ (t)\Big|& \leq \big(|\theta^\prime|+\big|\tfrac{\lambda^\prime}{\lambda}\big| \big) \|g\|_{\mathcal{E}}+ \big(|\zeta|^\prime +\big|\tfrac{\mu^\prime}{\mu}\big|\big) \lambda^{\frac{N+4}{2}} \\
&\qquad
+\frac{1}{\lambda^4}\Big(\lambda^{\frac{N-4}{2}}+ \lambda^{\frac{N}{4}} \|g\|_{\mathcal{E}}+
\|g\|_{\mathcal{E}}^2 \Big)\\
&\ll  \frac{1}{\lambda^4} |t|^{-\frac{N}{2(N-12)}}.
\end{align*}
\end{proof}

\section{Coercivity of the energy}

Our aim in this section is to get an estimate of the energy of $u= e^{i\zeta} W_\mu + e^{i\theta } W_\lambda +g$ under essential the same assumption as in the previous section. More precisely, we impose that  
\begin{equation}
\label{smallpara}
|\zeta + \pi /2| +|\mu -1| +|\theta | +\lambda +\|g\|_{\mathcal{E}} \ll 1,
\end{equation}
and $g$ satisfying \eqref{ortg}. We recall that
$$a_1^+ = \langle \alpha_{\zeta ,\mu}^+ ,g\rangle , a_1^- = \langle \alpha_{\zeta ,\mu}^- ,g\rangle,\ a_2^+ = \langle \alpha_{\theta ,\lambda}^+ ,g\rangle , a_2^- = \langle \alpha_{\theta ,\lambda}^- ,g \rangle .$$
The main result of this section will be the following proposition.
\begin{prop}
\label{ener}
 There exist constants $\eta ,C_0,C>0$ such that for all $u\in \mathcal{E}$ of the form $u= e^{i\zeta} W_\mu + e^{i\theta } W_\lambda +g$,
with 
\begin{equation}\label{eta}
|\zeta + \pi /2| +|\mu -1| +|\theta | +\lambda +\|g\|_{\mathcal{E}}\le\eta
\end{equation}  
and $g$ satisfying \eqref{ortg},
  we have
\begin{equation}
\label{propenere1}
|E(u)-2E(W)|\leq C \left((|\zeta +\pi /2| +|\mu -1| +|\theta| +\lambda )\lambda^{\frac{N-4}{2}} +\|g\|_{\mathcal{E}}^2 \right), 
\end{equation}
and
\begin{align}
\label{propenere2}
\|g\|_{\mathcal{E}}^2 &+C_0 \theta \lambda^{\frac{N-4}{2}} \leq C \Big(\lambda^{(N-4)/2} \big(|\zeta +\pi /2| +|\mu -1|+|\theta|^3 +\lambda\big) \nonumber\\
&+E(u) - 2E(W) +\sum_{j=1}^2 \big((a_j^+)^2 + (a_j^- )^2 \big)\Big).
\end{align}
\end{prop}

We decompose the proof into several lemma. At first, notice that by \eqref{fe6} we obtain a Taylor's expansion 
\begin{align}
\label{enere1}
&\Big|E(u) - E (e^{i\zeta} W_\mu +e^{i\theta} W_\lambda) - \langle DE (e^{i\zeta} W_\mu +e^{i\theta} W_\lambda ),g \rangle\\
&\qquad - \tfrac{1}{2} \langle D^2 E (e^{i\zeta} W_\mu +e^{i\theta} W_\lambda) g,g \rangle \Big|\nonumber\\
& \lesssim \|g\|_{\mathcal{E}}^{\frac{2N}{N-4}}.\nonumber
\end{align}
We begin by estimating the term $ E (e^{i\zeta} W_\mu +e^{i\theta} W_\lambda)  $.
\begin{lem}
\label{lemener1}
Under the assumptions of Proposition \ref{ener}, there exists a constant $C$ depending only on $N$ such that
\begin{align*}
&\left|E(e^{i\zeta} W_\mu +e^{i\theta} W_\lambda) - 2E(W) - C_1\theta  \lambda^{\frac{N-4}{2}}   \right|\\
& \leq C \lambda^{\frac{N-4}{2}} (|\zeta +\pi /2| +|\mu -1| +|\theta|^3 +\lambda),
\end{align*}
where
\[
C_1=\int_{\R^N} W^{\frac{N+4}{N-4}} dx.
\]
\end{lem}
\begin{proof}
Expanding the energy, we find
\begin{align*}
E(e^{i\zeta} W_\mu +e^{i\theta} W_\lambda)) &= E(e^{i\zeta} W_\mu ) +E(e^{i\theta} W_\lambda) +\mathcal{R} \int_{\R^N} e^{i(\zeta - \theta)} \Delta W_\mu \Delta W_\lambda dx \\
& - \int_{\R^N} \big[F(e^{i\zeta} W_\mu + e^{i\theta} W_\lambda) -F(e^{i\zeta} W_\mu) - F(e^{i\theta} W_\lambda) \big] dx.
\end{align*}
By scaling invariance, we have $E(e^{i\zeta} W_\mu) =E(W)$. Integrating by parts, we get
$$\mathcal{R} \int_{\R^N} e^{i(\zeta -\theta ) } \Delta W_\mu \Delta W_\lambda dx =  \mathcal{R} \int_{\R^N} \overline{e^{i\theta} W_\lambda} f(e^{i\zeta} W_\mu) dx.$$
So, inserting this identity into the expansion of the energy, we have
\begin{align*}
E(e^{i\zeta} & W_\mu +e^{i\theta} W_\lambda)= 2E(W) \\
&-\int_{\R^N} \big[F (e^{i\zeta} W_\mu +e^{i\theta} W_\lambda) -F (e^{i\zeta} W_\mu) - F (e^{i\theta} W_\lambda) - \mathcal{R} ( \overline{e^{i\theta} W_\lambda } f(e^{i\zeta} W_\mu)\big] dx. 
\end{align*}
When $|x|\geq \sqrt{\lambda}$, using 
\eqref{fe5}, we get 
$$|F (e^{i\zeta} W_\mu +e^{i\theta} W_\lambda) -F (e^{i\zeta} W_\mu) - F (e^{i\theta} W_\lambda) - \mathcal{R} ( \overline{e^{i\theta} W_\lambda } f(e^{i\zeta} W_\mu))|\lesssim W_\lambda^2 .$$
Recalling that $W_\lambda= \alpha_N \lambda^{-\frac{N-4}{2}} (1+|\frac{x}{\lambda}|^2)^{- \frac{N-4}{2}}$, we have
$$\int_{|x|\geq \sqrt{\lambda}} W_\lambda^2 dx = \lambda^4 \int_{|x|\geq \frac{1}{\sqrt{\lambda}} } W^2 dx\lesssim \lambda^4 \int_{1/\sqrt{\lambda}}^\infty r^{8-2N}r^{N-1} dx \lesssim \lambda^{4+ \frac{N-8}{2}}= \lambda^{\frac{N}{2}}.   $$
On the other hand, when $|x|\leq \sqrt{\lambda}$, we have $|\mathcal{R} ( \overline{e^{i\theta} W_\lambda } f(e^{i\zeta} W_\mu)| \lesssim W_\lambda $ and

$$\int_{|x|\leq \sqrt{\lambda}} W_\lambda dx \lesssim \lambda^{\frac{N}{2} +2} \int_0^{1/\sqrt{\lambda}} r^{4-N}r^{N-1} dr \lesssim \lambda^{\frac{N}{2}} .$$
Proceeding in the same way, we can show that the term $F(e^{i\zeta}W_\mu)$ is also negligible. Taylor's expansion \eqref{fe5} gives
$$|F(e^{i\zeta} W_\mu + e^{i\theta} W_\lambda) - F (e^{i\theta} W_\lambda) -\mathcal{R} (  \overline{e^{i\zeta} W_\mu } f(e^{i\theta} W_\lambda) |\lesssim W_\lambda^{\frac{8}{N-4}}.$$
Notice that
$$\int_{|x|\leq \sqrt{\lambda}} W_\lambda^{\frac{8}{N-4}} dx \lesssim \lambda^{N-4} \int_0^{\frac{1}{\sqrt{\lambda}}}  r^{-8} r^{N-1} dr \lesssim \lambda^{N-4 -\frac{N-8}{2}}=\lambda^{\frac{N}{2}}.$$
At this point, we have shown that
$$\left| E(e^{i\zeta} W_\mu +e^{i\theta} W_\lambda )- 2E(W) - \int_{|x|\leq \sqrt{\lambda}} \mathcal{R} (  \overline{e^{i\zeta} W_\mu } f(e^{i\theta} W_\lambda))  dx \right|\lesssim \lambda^{\frac{N}{2}}. $$
So we only need to estimate $ \int_{|x|\leq \sqrt{\lambda}} \mathcal{R} (  \overline{e^{i\zeta} W_\mu } f(e^{i\theta} W_\lambda))  dx $. To do so, observe that
\begin{align*}
&\left|\int_{|x|\leq \sqrt{\lambda}} \mathcal{R} ( \overline{e^{i\zeta} W_\mu } f(e^{i\theta} W_\lambda ) )dx - \mathcal{R} (e^{i(\zeta - \theta)})\int_{\R^N} W_\lambda^{\frac{N+4}{N-4}} dx \right|\\
&\lesssim \int_{|x|\leq \sqrt{\lambda}} |W_\mu -1| W_\lambda^{\frac{N+4}{N-4}} dx + \int_{|x|\geq \sqrt{\lambda}} W_\lambda^{\frac{N+4}{N-4}} dx \\
&\lesssim (|\mu -1|+\lambda) \int_{|x|\leq \sqrt{\lambda}} W_\lambda^{\frac{N+4}{N-4}} dx +\int_{|x|\geq \sqrt{\lambda}} W_\lambda^{\frac{N+4}{N-4}} dx\\
&\lesssim  (|\mu -1|+\lambda) \lambda^{\frac{N-4}{2}}.
\end{align*}
Since
$$|\mathcal{R} (e^{i (\zeta -\theta )} )+\theta| \leq |\mathcal{R} (-i e^{-i\theta} ) +\theta|+ |e^{i(\zeta -\theta) } +i e^{-i\theta}|\lesssim  |\theta|^3 +|\zeta + \pi /2|$$
and
$$\int_{\R^N} W_\lambda^{\frac{N+4}{N-4}} dx = C_1 \lambda^{\frac{N-4}{2}},$$
we deduce that
\begin{equation}
\label{missingeq}
\left|\int_{|x|\leq \sqrt{\lambda}} \mathcal{R} ( \overline{e^{i\zeta} W_\mu } f(e^{i\theta} W_\lambda ) dx - C_1 \theta \lambda^{\frac{N-4}{2}} \right|\lesssim \lambda^{\frac{N-4}{2}} (|\zeta +\pi /2| +|\mu -1| +|\theta|^3 +\lambda). 
\end{equation}
\end{proof}

\begin{lem}
\label{lemener2}
We have
$$|\langle DE (e^{i\zeta} W_\mu + e^{i\theta} W_\lambda),g \rangle| \lesssim \lambda^{\frac{N+4}{4}} \|g\|_{\mathcal{E}} .$$
\end{lem}

\begin{proof}
Since $DE (e^{i \zeta} W_\mu) = DE (e^{i\theta} W_\lambda)=0$, we want to prove that
$$|\langle f(e^{i\zeta} W_\mu + e^{i\theta} W_\lambda) - f( e^{i\zeta} W_\mu ) -f( e^{i\theta} W_\lambda) ,g \rangle | \lesssim \lambda^{\frac{N+4}{4}} \|g\|_{\mathcal{E}}. $$
By the Sobolev's embedding $\dot{H}^{2} \hookrightarrow
L^{\frac{2N}{N-4}} $, this reduces to show that
$$ \| f(e^{i\zeta} W_\mu + e^{i\theta} W_\lambda) - f( e^{i\zeta} W_\mu ) -f( e^{i\theta} W_\lambda) \|_{L^{\frac{2N}{N+4}}} \lesssim \lambda^{\frac{N+4}{4}} .$$
When $|x|\leq \sqrt{\lambda}$, we have $W_\mu \lesssim W_\lambda$ so we deduce that
\begin{align*}
&| f(e^{i\zeta} W_\mu + e^{i\theta} W_\lambda) - f( e^{i\zeta} W_\mu ) -f( e^{i\theta} W_\lambda)| \\
&\lesssim  W_\lambda^{\frac{8}{N-4}} W_\mu +W_\mu^{\frac{N+4}{N-4}}\\
& \lesssim   W_\lambda^{\frac{8}{N-4}} W_\mu \lesssim W_\lambda^{\frac{8}{N-4}}. 
\end{align*}
We obtain
$$\|W_\lambda^{\frac{8}{N-4}} \|_{L^{\frac{2N}{N+4}} (|x|\leq \sqrt{\lambda}) } \lesssim \lambda^{\frac{N-4}{2}} \left(\int_0^{\frac{1}{\sqrt{\lambda}}} r^{-\frac{16N }{N+4}}  r^{N-1} dr\right)^{\frac{N+4}{2N}} \lesssim \lambda^{\frac{N+4}{4}} .$$
On the other hand, when $|x|\geq \sqrt{\lambda}$, we have $W_\lambda \leq W_\mu$ so
\begin{align*}
&| f(e^{i\zeta} W_\mu + e^{i\theta} W_\lambda) - f( e^{i\zeta} W_\mu ) -f( e^{i\theta} W_\lambda)|\\
& \lesssim  W_\mu^{\frac{8}{N-4}} W_\lambda +W_\lambda^{\frac{N+4}{N-4}}\\
& \lesssim  W_\mu^{\frac{8}{N-4}} W_\lambda \lesssim W_\lambda. 
\end{align*}
We get
$$\|W_\lambda\|_{L^{\frac{2N}{N+4} } (|x|\geq \sqrt{\lambda}) } \lesssim \lambda^4  \left(\int_{\frac{1}{\sqrt{\lambda}}}^\infty r^{- \frac{2N (N-4)}{N+4}}  r^{N-1}dr \right)^{\frac{N+4}{2N}} \lesssim \lambda^{\frac{N+4}{4}}. $$
This concludes the proof.
\end{proof}

In order to deal with the quadratic term of \eqref{enere1}, we analyse the coercivity of $L^\pm$.

\begin{lem}
\label{lemcoer}
There exist constants $c,C>0$ such that for any real valued $g\in \mathcal{E}$, we have
\begin{equation}
\label{lemcoere1}
\langle g,L^+ g\rangle \geq c \int_{\R^N} |\Delta g|^2 dx - C (\langle W,g\rangle^2 + \langle Y^{(2)} ,g \rangle^2 ),
\end{equation}
and
$$\langle g,L^- g\rangle \geq c \int_{\R^N} |\Delta g|^2 dx - C \langle \Lambda W,g\rangle^2 .$$
If $r_1>0$ is large enough, then
\begin{align}
\label{coercloc1}
&(1-2c) \int_{|x|\leq r_1} |\Delta g|^2 dx +c \int_{|x|\geq r_1} |\Delta g|^2 dx - \frac{N+4}{N-4} \int_{\R^N}   W^{8/(N-4)} |g|^2 dx \nonumber \\
&\geq - C (\langle W,g\rangle^2 + \langle Y^{(2)} ,g \rangle^2).
\end{align}
\end{lem}

\begin{proof}
In a first time, we show that if $g\in \mathcal{E}$ is such that $\langle Y^{(2)} ,g\rangle =0$ then $\langle g ,L^+ g \rangle \geq 0$. We proceed by contradiction. Let $a,b\neq 0 \in \R$. Since $Y^{(2)} \neq W$, we have, by \eqref{defYs}, 
$$\langle Y^{(1)} , L^+ Y^{(1)}\rangle =- \langle L^- Y^{(2)} , Y^{(2)} \rangle <0.$$
This implies that
$$\langle ag +b Y^{(1)} , L^+ (ag + b Y^{(1)})  \rangle =a^2 \langle g, L^+ g \rangle - 2 ab \nu \langle g, Y^{(2)} \rangle +b^2 \langle Y^{(1)} ,L^+ Y^{(1)} \rangle <0 .$$
This is impossible since $L^+$ has only one negative direction. 

 Suppose that \eqref{lemcoere1} is false. Then there exists a sequence $g_n \in \mathcal{E}$ such that $\|g_n\|_{\mathcal{E}}=1$ and 
$$\langle g_n ,L^+ g_n\rangle \leq c_n - C_n  (\langle W,g_n\rangle^2 + \langle Y^{(2)} ,g_n \rangle^2 ),\ c_n \rightarrow 0,\ C_n\rightarrow \infty.$$
We can assume that $g_n \rightharpoonup g\in \mathcal{E}$. Since $|\langle g_n ,L^+ g_n \rangle| \lesssim \|g_n\|^2_{\mathcal{E}} = 1$, we deduce that
$\langle W,g\rangle = \langle Y^{(2)} ,g \rangle =0$. We also have $\langle g_n ,V^+ g_n \rangle \rightarrow \langle g ,V^+ g \rangle$ hence
$$\langle g, L^+ g \rangle \leq \liminf_n \langle g_n ,L^+ g_n\rangle \leq \liminf_n c_n =0.$$
So, we get that $\langle h, L^+ g \rangle =0$, for all $h\in \mathcal{E}$ such that $\langle h, Y^{(2)}\rangle =0$. This implies that $\langle Y^{(1)} , L^+ g \rangle =-\nu \langle Y^{(2)} , g\rangle =0$. Since $\langle Y^{(1)} , Y^{(2)} \rangle \neq 0$, we in fact get that $\langle h,L^+ g\rangle =0$ for all $h\in \mathcal{E}$. This implies that $g= \Lambda W$ but since $\langle W, \Lambda W \rangle \neq 0$ so we get a contradiction.\\
Next, we prove \eqref{coercloc1}. Let $\varphi$ be a cutoff function such that $\varphi \equiv 1$ in $B_{r_1}$ and with support on $B_{r_1^2}$ for $r_1$ large enough so that $\|\nabla^{(i)} \varphi \|_{L^\infty} \ll r_1^{-i}$, for $i=1,2$. Integrating by parts, we find
$$\int_{\R^N} |\Delta (\varphi u)|^2 dx = \int_{\R^N}[ \varphi^2 |\Delta u|^2 dx + u^2 |\Delta \varphi|^2 + 4 |\nabla\varphi|^2  |\nabla u|^2 ]dx.    $$
Using the properties of $\varphi$ and Hardy's inequality, we get that, for any $c$, 
$$\int_{\R^N}  u^2 |\Delta \varphi|^2 dx\leq \|\nabla^{(2)} \varphi \|_{L^\infty}^2 r_1^4 \int_{|x|\geq r_1}  \dfrac{u^2}{|x|^4} dx \leq c \int_{|x|\geq r_1} |\Delta u|^2 dx. $$
We can deal with the term $\int_{\R^N} |\nabla\varphi|^2  |\nabla u|^2 dx$ in the same way. Denoting by $V(x)=W^{8/(N-4)}(x) $ and noticing that $V\in L^{\frac{N}{4}}(\R^N)$ since $V(x) \approx |x|^{-8}$, we have, using H\" older's and Sobolev's inequalities,
\begin{align*}
\int_{\R^N} V g^2 dx& \leq  \int_{\R^N} V (\varphi g)^2 + \left(\int_{|x|\geq r_1} V^{\frac{N}{4}}dx \right)^{\frac{4}{N}}   \left(\int_{\R^N}  g^{\frac{2N}{N-4}} dx\right)^\frac{N-4}{N}\\
&\leq   \int_{\R^N} V (\varphi g)^2 + c \int_{|x|\geq r_1} |\Delta u|^2 dx .
\end{align*}
We use the same kind of estimate for the last two terms. We have $\langle  Y , \varphi g\rangle^2 \leq \langle  Y , g\rangle^2 + \langle  Y , \varphi g -g \rangle^2$. To control the last term, in the same way as we previously did, we use H\" older's and Sobolev's inequalities to deduce that
$$\langle  Y , \varphi g -g \rangle^2 \leq  c \int_{|x|\geq r_1} |\Delta u|^2 dx  ,$$
 since $Y\in L^{2N/(N+4)}(\R^N)$ (the dual of the critical $H^2$ exponent).
\end{proof} 

We use the previous lemma to study the coercivity of the linearization around $e^{i\theta} W_\lambda$.
\begin{prop}
\label{propcoer111}
There exist constants $c,C>0$ such that for any $\theta \in \R$, $\lambda >0$ and for any complex valued radial $g\in \mathcal{E}$, there holds
\begin{align*}
&\int_{\R^N} |\Delta g|^2 dx - \mathcal{R} \int_{\R^N} \bar{g} f^\prime (e^{i\theta} W_\lambda) g dx\\
& \geq c \int_{\R^N} |\Delta g|^2 dx - C (\langle \lambda^{-2} e^{i\theta} W_\lambda , g\rangle^2 + \langle \lambda^{-2} i e^{i\theta} \Lambda W_\lambda ,g \rangle^2 +\langle \alpha_{\theta ,\lambda}^+ ,g \rangle^2 +\langle \alpha_{\theta ,\lambda}^- ,g\rangle^2)
\end{align*}
If $r_1 >0$ is large enough, then we have
\begin{align*}
&(1-2c)\int_{|x|\leq r_1} |\Delta g|^2 dx+c \int_{|x_1| \geq r_1} |\Delta g|^2 dx - \mathcal{R} \int_{\R^N} \bar{g} f^\prime (e^{i\theta} W_\lambda) g dx\\
& \geq c \int_{\R^N} |\Delta g|^2 dx - C (\langle \lambda^{-2} e^{i\theta} W_\lambda , g\rangle^2 + \langle \lambda^{-2} i e^{i\theta} \Lambda W_\lambda ,g \rangle^2 +\langle \alpha_{\theta ,\lambda}^+ ,g \rangle^2 +\langle \alpha_{\theta ,\lambda}^- ,g\rangle^2)
\end{align*}
If $r_2>0$ is small enough, we have
\begin{align*}
&(1-2c)\int_{|x|\geq r_2} |\Delta g|^2 dx+c\int_{|x|\leq r_2} |\Delta g|^2 dx - \mathcal{R} \int_{\R^N} \bar{g} f^\prime (e^{i\theta} W_\lambda) g dx\\
& \geq c \int_{\R^N} |\Delta g|^2 dx - C (\langle \lambda^{-2} e^{i\theta} W_\lambda , g\rangle^2 + \langle \lambda^{-2} i e^{i\theta} \Lambda W_\lambda ,g \rangle^2 +\langle \alpha_{\theta ,\lambda}^+ ,g \rangle^2 +\langle \alpha_{\theta ,\lambda}^- ,g\rangle^2)
\end{align*}

\end{prop}
 \begin{proof}
We can assume without loss of generality that $\theta=0$ and $\lambda =1$. Let $g=g_1 +i g_2$. We have
$$f^\prime (W) (g_1 +ig_2)=W^{\frac{8}{N-4}} \left(\frac{N+4}{N-4} g_1 +i  g_2 \right).$$
So there holds $\mathcal{R} \int_{\R^N} \bar{g} f^\prime (W) g dx = \int_{\R^N } W^{\frac{8}{N-4}} \big(\frac{N+4}{N-4}  g_1^2 + g_2^2\big) dx.$ We have $\langle W,g\rangle = \langle W ,g_1\rangle$ and $\langle i \Lambda W,g\rangle = \langle \Lambda W ,g_2\rangle$. We have $Y^{(2)} = \frac{\alpha_{0,1}^+ +\alpha_{0,1}^- }{2}$, so
$$\langle Y^{(2)} ,g_1 \rangle^2 = \langle Y^{(2)} ,g \rangle^2 \leq \frac{\langle \alpha_{0,1}^+ ,g\rangle^2 +\langle \alpha_{0,1}^- ,g\rangle^2 }{2} $$
We get the result by applying the previous result.
\end{proof}

Next, we deduce a coercivity result for the linearization around our sum of two bubbles.
\begin{lem}
\label{lemener3}
There exist $\eta ,C>0$ such that if $\lambda \leq \eta \mu$, then for all $g\in \mathcal{E}$ satisfying \eqref{ortg} and \eqref{eta} 
 there holds
$$\frac{1}{C} \|g\|_{\mathcal{E}}^2 \leq \frac{1}{2} \langle D^2 E (e^{i\zeta}W_\mu +e^{i\theta} W_\lambda )g,g\rangle +2 \big((a_1^+)^2 + (a_1^-)^2 +(a_2^+)^2 +(a_2^-)^2\big)\leq C \|g\|_{\mathcal{E}}^2$$
\end{lem}
\begin{proof}
The proof of the upper bound is straight-forward since we deduce, from $\|\alpha_{\theta , \lambda}^\pm\|_{L^{\frac{2N}{N+4}}}\lesssim 1$, that $|a_i^\pm |\lesssim \|g\|_{\mathcal{E}}$. For the lower bound, we claim that, for any $c>0$, there holds
\begin{align*}
& | \langle D^2 E (e^{i\zeta}W_\mu +e^{i\theta} W_\lambda )g,g\rangle  -[ \int_{\R^N} |\Delta g|^2 -\mathcal{R}\int_{\R^N}\overline{g} g (f^\prime (e^{i\theta} W_\lambda) +f^\prime (e^{i\zeta} W_\mu ) ) dx] |  \\
&\lesssim c \|g\|^2_{\mathcal{E}}
\end{align*}
provided $\eta$ is small enough. The result then follows by using the previous lemma and the orthogonality conditions. Let us notice, using H\" older and Sobolev inequality that the claim reduces to show that
$$\|f^\prime (e^{i\theta} W_\lambda + e^{i\zeta} W_\mu )  - f^\prime (e^{i\theta} W_\lambda) - f^\prime (e^{i\zeta} W_\mu )  \|_{L^{\frac{N}{4}}(\R^N) }\leq c.  $$
We can assume without loss of generality that $\mu=1$. We decompose this integral into two parts depending whether $|x|\leq \sqrt{\lambda}$ or not. If $|x|\leq \sqrt{\lambda}$, then $W\lesssim W_\lambda$ and so
$$|f^\prime (e^{i\theta} W_\lambda + e^{i\zeta} W_\mu )  - f^\prime (e^{i\theta} W_\lambda) - f^\prime (e^{i\zeta} W_\mu )| \lesssim 1. $$
On the other hand, if $|x|\geq \sqrt{\lambda}$, we have
$$|f^\prime (e^{i\theta} W_\lambda + e^{i\zeta} W_\mu )  - f^\prime (e^{i\theta} W_\lambda) - f^\prime (e^{i\zeta} W_\mu )| \lesssim  |f^\prime (e^{i\theta} W_\lambda)|. $$
The claim then follows by a simple integration.
\end{proof}
We are finally in position to prove Proposition \ref{ener}.
\begin{proof}[Proof of Proposition \ref{ener}]
We begin by proving \eqref{propenere1}. By using successively Lemma \ref{lemener1}, \ref{lemener2}, \ref{lemener3} and \eqref{enere1}, we find
\begin{align*}
&|E(u)-2E(W)|\le \big|E(u) - E (e^{i\zeta}W_\mu +e^{i\theta} W_\lambda) + C_1 \theta \lambda^{\frac{N-4}{2}} \big|\\
& +C \lambda^{\frac{N-4}{2}} \big(|\zeta +\pi/2| +|\mu -1| +|\theta|^3 +\lambda\big)\\
&\lesssim  |E(u) - E (e^{i\zeta}W_\mu +e^{i\theta} W_\lambda) - \langle DE (e^{i\zeta}W_\mu +e^{i\theta} W_\lambda),g\rangle | \\
& +C \lambda^{\frac{N-4}{2}} \big(|\zeta +\pi/2| +|\mu -1| +|\theta| +\lambda\big)  +\lambda^{\frac{N+4}{4}}\|g\|_{\mathcal{E}} \\
&\lesssim  \big|E(u) - E (e^{i\zeta}W_\mu +e^{i\theta} W_\lambda) - \langle DE (e^{i\zeta}W_\mu +e^{i\theta} W_\lambda),g\rangle - \tfrac{1}{2}\langle D^2 E (e^{i\zeta}W_\mu +e^{i\theta} W_\lambda )g,g\rangle \big| \\
& +C \lambda^{\frac{N-4}{2}} \big(|\zeta +\pi/2| +|\mu -1| +|\theta| +\lambda\big)  +\lambda^{\frac{N+4}{4}}\|g\|_{\mathcal{E}} +\|g\|_{\mathcal{E}}^2  \\
&\lesssim C \lambda^{\frac{N-4}{2}} \big(|\zeta +\pi/2| +|\mu -1| +|\theta| +\lambda\big) +\|g\|_{\mathcal{E}}^2 +\|g\|_{\mathcal{E}}^{\frac{2N}{N-4}}  .
\end{align*}
Notice that, for any $c>0$, we have $\|g\|_{\mathcal{E}}^{\frac{2N}{N-4}}\leq c \|g\|_{\mathcal{E}}^2$ if $\eta$ is small enough.
Next, we turn to the proof of \eqref{propenere2}. From Lemma \ref{lemener1}, \ref{lemener2} and \eqref{enere1}, we have
\begin{align*}
&|E(u) - 2E(W) -\theta C_1  \lambda^{\frac{N-4}{2}} - \tfrac{1}{2} \langle D^2 E (e^{i\zeta} W_\mu +e^{i\theta} W_\lambda) g,g \rangle |\\
& \leq C (|\zeta + \pi /2|+|\mu -1| +|\theta|^3 +\lambda) \lambda^{(N-4)/2}+c \|g\|_{\mathcal{E}}^2
\end{align*}
so
\begin{align*}
&\theta C_1 \lambda^{\frac{N-4}{2}} +\tfrac{1}{2} \langle D^2 E (e^{i\zeta} W_\mu +e^{i\theta} W_\lambda) g,g \rangle \\
& \leq E(u) - 2E(W) +C (|\zeta + \pi /2|+|\mu -1| +|\theta|^3 +\lambda) \lambda^{(N-4)/2}+c \|g\|_{\mathcal{E}}^2    .
\end{align*}
Therefore, choosing $c>0$ small enough, the result follows from Lemma \ref{lemener3}.
\end{proof}


\section{Bootstrap procedure: proof of Theorem \ref{main}}
The aim of this section is basically to prove that the assumptions we used in Proposition \ref{lem3.1} indeed hold true. Once it is done, Theorem \ref{main} will follow by using a Brower fixed point argument. In a first time, we will choose a 'good' initial data at time $T$, namely we take $u(T)=-iW +W_{\lambda_0}+g_0$ where $\lambda_0 \approx |T|^{-\frac{2}{N-12}} $ and $g_0$ satisfies appropriate orthogonality conditions. 
\begin{lem}
\label{lemcondinit}
There exists $T_0<0$ such that for all $T\leq T_0$ and for all $\lambda^0$, $a_1^0$, $a_2^0$ satisfying
\begin{equation}
\label{condinit}
|\lambda^0 - \tilde{C} |T|^{-\frac{2}{N-12}}|\leq \frac{1}{2} |T|^{-\frac{5}{2(N-12)}},\ |a_1^0|\leq \frac{1}{2} |T|^{-\frac{N}{2(N-12)}},\ |a_2^0|\leq \frac{1}{2} |T|^{-\frac{N}{2(N-12)}},
\end{equation}
there exists $g^0 \in X^2 = \dot{H}^3 (\R)^N \cap \dot{H}^2 (\R^N)$ satisfying
\begin{equation}\label{4-2}
\langle \Lambda W , g^0\rangle = \langle iW , g^0\rangle = \langle i \Lambda W_{\lambda^0},g^0 \rangle = \langle - W_{\lambda^0}, g^0\rangle =0,
\end{equation}
\begin{equation}\label{4-3}
\langle \alpha^-_{-\frac{\pi}{2},1} , g^0 \rangle =0,\ \langle \alpha^+_{-\frac{\pi}{2},1}, g^0 \rangle = a_1^0 ,\ \langle \alpha_{0,\lambda^0}^- , g^0 \rangle =0 ,\ \langle \alpha_{0,\lambda^0}^+,g^0 \rangle = a_2^0,
\end{equation}
\begin{equation}
\label{eqg0}
\|g^0\|_{\mathcal{E}} \lesssim |T|^{-\frac{N}{2(N-12)}} .
\end{equation}
Moreover, $g^0$ is continuous in the $X^2$ topology with respect to the parameters $\lambda^0$, $a_1^0$ and $a_2^0$.
\end{lem}

\begin{proof}
Let $g^0$ be of the following form
\begin{align*}
g^0 = & a_1^+ i \alpha^-_{-\frac{\pi}{2},1} - a_1^- i \alpha^+_{-\frac{\pi}{2}, 1}+b_1 W +c_1 (-i \Lambda W)\\
& +a_2^+ (\lambda^0)^4 i \alpha_{0,\lambda^0}^- -a_2^- (\lambda^0)^4 i \alpha_{0,\lambda^0}^+ + b_2 i W_{\lambda^0} +c_2 \Lambda W_{\lambda^0}, 
\end{align*}
with $a_1^+ ,a_1^- , b_1 ,c_1 ,a_2^+ ,a_2^- , b_2 ,c_2$ being real numbers. Let $\Phi : \R^8 \rightarrow \R^8$ be the linear map defined as
\begin{align*}
&\Phi (a_1^+ , a_1^- , b_1 ,c_1 ,a_2^+ ,a_2^- , b_2 ,c_2) = \big(\langle \alpha^+_{-\frac{\pi}{2},1}, g^0 \rangle , \langle \alpha^-_{-\frac{\pi}{2},1}, g^0 \rangle , \langle \Lambda W , g^0 \rangle , \langle iW ,g^0 \rangle, \\
&\qquad \langle \alpha^+_{0,\lambda^0} , g^0 \rangle , \langle \alpha_{0,\lambda^0}^- , g^0 \rangle , \langle (\lambda^0)^{-4} i \Lambda W_{\lambda^0}, g^0 \rangle , \langle - (\lambda^0)^{-4} W_{\lambda^0} , g^0 \rangle \big).
\end{align*}
Hence we need to solve
\[
\Phi (a_1^+ , a_1^- , b_1 ,c_1 ,a_2^+ ,a_2^- , b_2 ,c_2) = (a_1^0 ,0 , 0,0,a_2^0 ,0 ,0,0).
\]
First we are going to prove that $\Phi$ is invertible.
Denoting by $(M_{ij})$, $i,j=1,\ldots ,8$, the matrix of $\Phi$, we see by direct computations that the diagonal entries satisfy
\[
|M_{ii}|\ge c>0
\]
for some constant $c$ independent of $\lambda^0$ whereas for the other entries we have
\[
|M_{ij}| \lesssim |\lambda^0|^{\frac{N-11}{2}},\quad i\ne j.
\]
Choosing $|T_0|$ large enough, $(M_{ij})$ will be strictly diagonally dominant, and therefore $\Phi$ is invertible. We also notice that $a_i^- \approx b_i \approx c_i \approx 0 $, $i=1,2$, $a_1^+ \approx a_1^0$ and $a_2^+ \approx a_2^0$. Since $\|\lambda^4 \alpha_{\theta , \lambda}\|_{\dot{H}^2} \lesssim 1$, we deduce that $\|g\|_{\mathcal{E}} \lesssim \max \{|a_1^+|, |a_2^+|\} \lesssim |T|^{-\frac{N}{2(N-12)}}$.

\end{proof}


In order to control $\theta^\prime$, we will use a localized virial functional as in \cite{JJ} (see also \cite{raphael2011existence}). The next technical lemma gives an explicite expression for the localizing function $q$. We require that this function satisfies some suitable conditions which will allow us to get \eqref{properA3}.

\begin{lem}
\label{lemconsq}
For any $c>0$ and $R>0$, there exists a radial function $q=q(c,R) \in C^{5,1} (\R^N)$ satisfying the following properties:
\begin{itemize}
\item[(1)] $q(x)= \frac{1}{2} |x|^2$ for $|x|\leq R$.
\item[(2)] There exists $\tilde{R}>0$ (depending on $c$ and $R$) such that $q(x) \equiv \text{Const}$ for $|x|\geq \tilde{R}$.
\item[(3)] $|\nabla q (x)|\lesssim |x|$ and $|\Delta q(x) | \lesssim 1 $ for all $x\in \R^N$, with constants independent of $c$ and $R$.
\item[(4)] $\partial_{rr}^2 q(x) \geq - c$ for all $x\in \R^N$. 
\item[(5)] $\big(2\partial_{rr}^2(\Delta q)+\Delta^2 q\big)(x) \leq c |x|^{-2}$ and $-\Delta^3 q (x) \leq c |x|^{-4} $  for all $x\in \R^N$.
\end{itemize}
\end{lem}
\begin{proof}
It is enough to assume that $R=1$ since the function $q_R:=R^2 q(\cdot/R)$ satisfies the properties if and only if $q$ does.
First, for $0<\varepsilon\ll 1$, we define a function $q_0=q_{0,\varepsilon}\colon [0,\infty)\to\R$ by setting
\[
q_0(s)=
\begin{cases}
\frac{1}{2}s^2, &\text{ if }s\le 1\\
c_0r^{2-\varepsilon}+c_1s + c_2+c_3s^{2-N} +c_4s^{4-N}+c_5s^{6-N}, &\text{ if }s\ge 1,
\end{cases}
\]
where
\begin{align*}
c_0&=\frac{N(N-2)(N-4)}{(\varepsilon-1)(\varepsilon-2)(N-\varepsilon)(N-2-\varepsilon)(N-4-\varepsilon)},\\
c_1&=\frac{\varepsilon N(N-2)(N-4)}{(\varepsilon-1)(N-1)(N-3)(N-5)},\\
c_2&=-\frac{\varepsilon N}{2(\varepsilon-2)(N-6)},\\
c_3&=-\frac{\varepsilon (N-4)}{8(N-\varepsilon)(N-1)},\\
c_4&=\frac{\varepsilon N}{4(N-2-\varepsilon)(N-3)},\\
\noalign{and}
c_5&=\frac{\varepsilon N(N-2)}{8(N-4-\varepsilon)(N-5)(N-6)}.
\end{align*}
We notice that \[
\sum_{i=0}^5 c_i=1/2,
\]  
\[
(c_0,c_1,c_2,c_3,c_4,c_5)\to (1/2,0,0,0,0,0)
\] 
as $\varepsilon\to 0+$, 
and that
\[
\lim_{s\to 1}\big(q_0(s),q^\prime_0(s),q^{\prime\prime}_0(s),q^{\prime\prime\prime}_0(s),q^{(4)}_0(s),q^{(5)}_0(s)\big)=(1/2,1,1,0,0,0).
\]
Furthermore, for $s>1$ and small enough $\varepsilon>0$, we have
\begin{align*}
q^\prime_0(s)&=(2-\varepsilon)c_0 s^{1-\varepsilon}+O(\varepsilon),\\
q^{\prime\prime}_0(s)&=(2-\varepsilon)(1-\varepsilon)c_0s^{-\varepsilon}+O\big(\varepsilon s^{4-N}\big)\ge 0,\\
\noalign{and}
q^{(j)}_0(s)&=O\big(\varepsilon s^{-(j-2)-\varepsilon}\big)\quad \text{ for }j\in\{3,4,5,6\}.
\end{align*}

Letting $r(x)=|x|$ denote the radial coordinate we notice that
$q_0:=q_0\circ r\in C^{5,1}(\R^N)$.
Furthermore, $q_0$ is convex since $q^{\prime\prime}_0\ge 0$, and for $|x|>1$
\begin{align*}
\Delta q_0(x)&=(2-\varepsilon)(N-\varepsilon)c_0r^{-\varepsilon} +O(\varepsilon r^{-1}),\\
\Delta^2 q_0(x)&= O\big(\varepsilon r^{-2-\varepsilon}\big),\\
\Delta^3 q_0(x)&=O\big(\varepsilon r^{-4-\varepsilon}\big),
\end{align*}
and
\[
\big(2\partial_{rr}^2(\Delta q_0)+\Delta^2 q_0\big)(x)=O\big(\varepsilon r^{-2-\varepsilon}\big).
\]

Let $\chi\in C^\infty_0\big([0,\infty)\big)$ be the standard cut-off function such that $\chi(r)=1$ for $0\le r\le 1$ and  $\chi(r)=0$ for $r\ge 2$. Let $R_0=R_0(c,\varepsilon)\gg 1$. Define $e_j(r)=(1/j!)r^j\,\chi(r)$ for $j\in\{1,2,3,4,5\}$ and
\[
q(r)=\begin{cases}
q_0(r), &\text{ if }r\le R_0;\\
q_0(R_0)+\sum_{j=1}^5 q_0^{(j)}(R_0)R_0^j e_j(-1+R_0^{-1}r), &\text{ if }r\ge R_0.
\end{cases}
\]
Note that $q_0^\prime(R_0)\approx R_0^{1-\varepsilon}$,  $q_0^{\prime\prime} (R_0)\approx R_0^{-\varepsilon}$, and 
$q_0^{(j)}(R_0)\approx \varepsilon R_0^{-j+2}$ for $j\in\{3,4,5\}$.
Clearly $q\in C^{5,1} (\R^N)$ and $q(r)=q_0(R_0)$ for $r\ge 3R_0$ by the definition of the functions $e_j$. Hence properties (1) and (2) hold. Furthermore, it follows from the definition of $q$ that, for $R_0\le r\le 3R_0$, $|q^\prime (r)|\lesssim R_0^{1-\varepsilon}\lesssim r$, $|q^{\prime\prime}(r)|\lesssim R_0^{-\varepsilon}\lesssim 1$, and $|q^{(j)}(r)|\lesssim R_0^{-(j-2)-\varepsilon}$ for $j\in\{3,4,5\}$. Hence property (3) holds. Similarly, property (4) holds if $R_0\gg 1$ is large enough. Finally, property (5) follows by choosing $\varepsilon>0$ small enough depending on $c$ and $R$. 

\end{proof}

To define our localized virial, we define
$$[A(\lambda) h](x)= \frac{N-4}{2N\lambda^4}\Delta q \big(\frac{x}{\lambda}\big)h(x)+ \frac{1}{\lambda^3}\nabla q \big(\frac{x}{\lambda}\big)\cdot \nabla h(x),$$
and
$$[A_0 (\lambda)h](x)= \frac{1}{2\lambda^4}\Delta q \big(\frac{x}{\lambda}\big) h(x)+ \frac{1}{\lambda^3} \nabla q \big(\frac{x}{\lambda}\big)\cdot \nabla h (x).$$
We will write $A$ and $A_0$ instead of $A(1)$ and $A_0 (1)$. For all $h\in \mathcal{E}$, notice that
\begin{equation}\label{scaling}
A(\lambda)(h_\lambda)= \lambda^{-4} (Ah)_\lambda,\ A_0 (\lambda) (h_\lambda)= \lambda^{-4} (A_0 h)_\lambda .
\end{equation}

\begin{lem}
\label{lemproperA}
\begin{itemize}
\item[(i)] For $\lambda>0$, the families $\{A(\lambda )\}, \{A_0 (\lambda) \}, \{\lambda \partial_\lambda A(\lambda)\}$, and $\{\lambda \partial_\lambda A_0 (\lambda)\}$ are bounded in $\mathcal{L} (\mathcal{E} ; \dot{H}^{-2})$ and the families $\{\lambda A(\lambda)\}$ and  $\{\lambda A_0 (\lambda)\}$ are bounded in $\mathcal{L} (\mathcal{E} ; L^2)$.
\item[(ii)] For any complex-valued function $h_1 ,h_2 \in X^2 (\R^N)$ and $\lambda>0$, we have
\begin{equation}
\label{properA1}
\langle A(\lambda)h_1 , f(h_1 +h_2) - f(h_1) - f^\prime (h_1) h_2\rangle = - \langle A(\lambda)h_2 , f(h_1 +h_2) - f(h_1) \rangle ,
\end{equation}
\begin{equation}
\label{properA2}
\langle h_1 ,A_0 (\lambda) h_2 \rangle =- \langle A_0 (\lambda ) h_1 ,h_2 \rangle .
\end{equation}
\item[(iii)] For any $c_0>0$, if we choose $c$ in Lemma \ref{lemconsq} small enough, then for all $h\in X^2$,
\begin{equation}
\label{properA3}
\langle A_0 (\lambda )h, -\Delta^2 h \rangle \leq \dfrac{c_0}{\lambda^4 } \|h\|^2_{\mathcal{E}} - \dfrac{1}{\lambda^4} \int_{|x|\leq R \lambda} |\Delta h(x)|^2 dx .
\end{equation}
\end{itemize}
\end{lem}
\begin{proof}
Claims (i) and (ii) follow, up to obvious modifications,  as in \cite[Lemma 4.7]{JJ}.

Let us establish \eqref{properA3}. Integrating by parts, see also Boulenger-Lenzmann
\cite[pp. 513-515]{bou-lenz} (notice that there is a misprint in the definition of $\Gamma_{\varphi_R}$ which should be $-i(2 \nabla \varphi_R .\nabla +\nabla. \nabla \varphi_R)$), we have, using that $q\in C^{5,1}(\R^N)$
\begin{align}\label{A3est}
\langle A_0 h ,-\Delta^2 h\rangle &= -2 \int_{\R^N}\left( \partial_{rr}^2 q |\partial_{rr}^2 h|^2 +\dfrac{\partial_r q}{r} \frac{N-1}{r^2}|\partial_r h|^2 \right) dx \\
&+ \int_{\R^N} |\nabla h |^2 \big(\partial^2_{rr} (\Delta q) +\frac{1}{2}\Delta^2 q\big) dx -\dfrac{1}{4} \int_{\R^N}  |h|^2 \Delta^3 q dx.\nonumber 
\end{align}
Then \eqref{properA3} follows from \eqref{A3est},  
Lemma \ref{lemconsq}, and Hardy's inequality. 

\end{proof}

We are now in position to prove the main bootstrap proposition. We will show that starting with the initial data given in Lemma \ref{condinit}, the parameters $\zeta$, $\mu$, $\theta$ and $g$ satisfy the assumption of Proposition \ref{lem3.1}. The main difficulty will be to control $\theta$.

\begin{prop}\label{4-4}
There exists $T_0<0$ with the following property. Let $T<T_1<T_0$ and let $\lambda^0$, $a_1^0$ and $a_2^0$ satisfy \eqref{condinit}. Let $g^0 \in X^2$ be given by Lemma \ref{lemcondinit} and consider the solution $u(t)$ to \eqref{eq} with initial data $u(T)=-iW+W_{\lambda_0}+g^0$. Suppose that $u(t)$ exists on $[T,T_1]$ and that \eqref{mode1}- \eqref{mode5} hold for any $t\in [T,T_1]$ as well as
\begin{equation}
\label{lastcondai}
|a_1^+ (t)| \leq |t|^{-\frac{N}{2(N-12)}},\ |a_2^+ (t)|\leq |t|^{-\frac{N}{2(N-12)}} .
\end{equation}
Then, for any $t\in [T,T_1]$, we have
\begin{equation}
\label{lastconde1}
\big|\zeta (t) +\frac{\pi}{2}\big| \leq \frac{1}{2} |t|^{-\frac{3}{N-12}},
\end{equation}
\begin{equation}
\label{lastconde2}
|\mu (t) -1|\leq \frac{1}{2} |t|^{-\frac{3}{N-12}},
\end{equation}
\begin{equation}
\label{lastconde3}
|\theta (t)|\leq \frac{1}{2} |t|^{-\frac{1}{N-12}},
\end{equation}
and
\begin{equation}
\label{lastconde4}
\|g(t)\|_{\mathcal{E}} \leq \frac{1}{2}|t|^{- \frac{N-3}{2(N-12)}}.
\end{equation}

\end{prop} 

\begin{proof}
Integrating \eqref{mode6} over $[T,t]$ with the initial value $\zeta (T)= -\frac{\pi}{2}$ we obtain
\begin{align*}
\left|\zeta (t)+\dfrac{\pi}{2}\right|&= |\zeta (t) - \zeta (T)| =
\left|\int_T^t \zeta^\prime (s) ds\right|\leq c\int_T^t  |s|^{-\frac{N-5}{N-12}} ds \le
 \frac{c(N-12)}{7}|t|^{-\frac{7}{N-12}}\\
& \le\frac{1}{2}|t|^{-\frac{3}{N-12}}. 
\end{align*}
The estimate \eqref{lastconde2} follows similarly.

Next we show that
\begin{equation}\label{a1_a2_}
|a_1^-(t)|<|t|^{-\frac{N}{2(N-12)}}\quad\text{and}\quad
|a_1^-(t)|<|t|^{-\frac{N}{2(N-12)}}
\end{equation}
for $t\in [T,T_1]$. The claims hold for $t=T$ by \eqref{4-3}.
Let $T_2\in (T,T_1)$ be the last time for which \eqref{a1_a2_} holds for $t\in [T,T_2)$. Suppose, for example, that $a_1^-=|T_2|^{-\frac{N}{2(N-12)}}$. 
Then \eqref{stcoe2} implies that $\frac{d}{dt}a_1^-(T_2)<0$ provided $c<\nu$. This contradicts the assumption that $a_1^-<|T_2|^{-\frac{N}{2(N-12)}}$
for $t<T_2$ (note that $T_2<T_0<0$). The other inequality in \eqref{a1_a2_} follows similarly.

Next we will prove that, for all $c_0>0$,
\begin{equation}\label{(4-18)}
|\theta(t)|\le c_0|t|^{-\frac{1}{N-12}}
\end{equation}
for $t\in [T,T_1]$ provided $|T_0|$ is chosen large enough depending on $c_0$.
By \eqref{propenere1}, \eqref{condinit}, \eqref{eqg0} and the conservation of the energy, we have
\begin{align*}
|E(u)- &2E(W)|=|E(u(T))-2E(W)|\\
&\lesssim \big(|\zeta (T) +\pi /2| +|\mu (T) -1| +|\theta (T)| +\lambda (T) \big)\lambda (T)^{\frac{N-4}{2}} +\|g(T)\|_{\mathcal{E}}^2\\
&=(\lambda^0)^{\frac{N-2}{2}}+\|g^0\|_{\mathcal{E}}^2\\
&\lesssim |T|^{-\frac{N-2}{N-12}} \le |t|^{-\frac{N-2}{N-12}}.
\end{align*}
So, using \eqref{propenere2}, we get
\begin{align*}
\theta \lambda^{\frac{N-4}{2}}& \lesssim \big(|\zeta +\pi/2|+|\mu -1| +|\theta|^3 +\lambda\big) \lambda^{\frac{N-4}{2}}
+|t|^{-\frac{N-2}{N-12}}+|t|^{-\frac{N}{N-12}}\\
&\lesssim |t|^{-\frac{N-2}{N-12}}.
\end{align*}
We deduce that
\begin{equation}\label{4-19}
\theta(t) \lesssim |t|^{-\frac{2}{N-12}}
\ll |t|^{-\frac{1}{N-12}}.
\end{equation}
In order to prove a converse estimate
\begin{equation}\label{4-20}
\theta(t)\ge -c_0|t|^{-\frac{1}{N-12}}
\end{equation}
we define
\[
\psi(t):=\theta (t) - \dfrac{1}{4\|W\|_{L^2}^2} \langle g(t), i A_0 (\lambda (t)) g(t)\rangle.
\]
We claim that 
\begin{equation}\label{4-21}
\psi^\prime (t)\ge -c_1|t|^{-\frac{N-11}{N-12}}
\end{equation}
for $t\in [T,T_1]$ with the constant $c_1>0$ as small as we wish by taking $|T_0|$ large enough.
Taking $|T_0|$ large enough and choosing $c=c_1/4$ in Proposition~\ref{lem3.1} we obtain, by using \eqref{mode9} and \eqref{4-19}, 
\begin{align}\label{4-22}
\psi^\prime &\geq 
-\frac{C_2}{2\|W\|_{L^2}^2}\theta\lambda^{\frac{N-12}{2}}
+\frac{K}{2\lambda^4 \|W\|_{L^2}^2} -\frac{c_1}{4}|t|^{-\frac{N-11}{N-12}} -\frac{1}{4\|W\|_{L^2}^2}\frac{d}{dt} \langle g, i A_0 (\lambda) g\rangle \nonumber\\
&\geq \frac{1}{2\|W\|_{L^2}^2} \left(\frac{K}{\lambda^4} - \frac{1}{2} \frac{d}{dt} \langle g, iA_0 (\lambda ) g \rangle \right) - \frac{c_1}{2}|t|^{-\frac{N-11}{N-12}}.
\end{align}
Hence we need to compute $ \frac{1}{2} \frac{d}{dt} \langle g, iA_0 (\lambda ) g \rangle$ up to terms of order $\ll |t|^{-\frac{N-11}{N-12}}$.
Using that $iA_0 (\lambda)$ is symmetric, we find
\begin{equation}\label{4-23}
\frac{1}{2} \frac{d}{dt} \langle g, iA_0 (\lambda ) g \rangle  = \frac{1}{2}\lambda^\prime \langle g, i\partial_\lambda A_0 (\lambda) g \rangle + \langle \partial_t g , iA_0 (\lambda) g\rangle.
\end{equation}
Since $\|i\lambda \partial_\lambda A_0 (\lambda) g \|_{\dot{H}^{-2}} \lesssim \|g\|_{\mathcal{E}}$, we get that
$$|\lambda^\prime \langle g, i\partial_\lambda A_0 (\lambda) g \rangle| \lesssim \big|\frac{\lambda^\prime}{\lambda}\big| \|g\|_{\mathcal{E}}^2
\ll |t|^{-\frac{N-11}{N-12}} .$$
To estimate the second term in \eqref{4-23} we write
 $\partial_t g$ as in \eqref{eqdeg}. First we estimate 
 $$C_3:=\langle -  \zeta^\prime i e^{i\zeta} W_\mu + \frac{\mu^\prime}{\mu} e^{i\zeta}\Lambda W_\mu -\theta^\prime i e^{i\theta} W_\lambda + \frac{\lambda^\prime}{\lambda} \Lambda W_\lambda , iA_0(\lambda) g \rangle . $$
Since $|\zeta^\prime (t)| +|\frac{\mu^\prime (t)}{\mu (t)}| + |\theta^\prime (t)|+ |\frac{\lambda^\prime (t)}{\lambda (t)}|\lesssim |t|^{-1} $, and $\|A_0 (\lambda) g\|_{\dot{H}^{-2}}\lesssim \|g\|_{\mathcal{E}}$, we get
$$|C_3|\lesssim |t|^{-1}\|g\|_{\mathcal{E}}
\le |t|^{-\frac{3N-27}{2(N-12)}}\ll |t|^{-\frac{N-11}{N-12}}. $$
Hence we have
\begin{align*}
&
\frac{1}{2} \frac{d}{dt} \langle g, iA_0 (\lambda ) g \rangle \\
&\quad= \langle -\Delta^2 g + f(e^{i\zeta} W_\mu +e^{i\theta}W_\lambda +g)- f(e^{i\zeta} W_\mu) - f(e^{i\theta} W_\lambda)  , A_0 (\lambda )g\rangle\\
&\qquad
+o\big(|t|^{-\frac{N-11}{N-12}}\big) .
\end{align*}
Next we notice that the function $A_0 (\lambda ) g$ is supported in the ball of radius $\tilde{R}\lambda$ and in this ball, we have $W_\lambda \ll W_\mu$. Consequently,  
$|f(e^{i\zeta} W_\mu +e^{i\theta}W_\lambda) - f(e^{i\zeta}W_\mu)- f (e^{i\theta}W_\lambda)|\lesssim |W|_\lambda^{\frac{8}{N-4}}$ by Taylor's expansion \eqref{fe2}. On the other hand, we have
$$\|W_\lambda^{\frac{8}{N-4}}\|_{L^2 (|x|\leq \tilde{R} \lambda)}=\lambda^{\frac{N-4}{2}} \|W^{\frac{8}{N-4}}\|_{L^2 (|x|\leq \tilde{R})}\lesssim \lambda^{\frac{N-4}{2}}\lesssim |t|^{-\frac{N-4}{N-12}},$$
and, by Lemma \ref{lemproperA} (i),
 $\|A_0 (\lambda) g\|_{L^2} \lesssim \lambda^{-1}\|g\|_{\mathcal{E}}.$
So, we deduce from the three previous estimates and the Cauchy-Schwarz inequality that
\begin{align*}
\frac{1}{2} \frac{d}{dt} \langle g, i A_0 (\lambda) g\rangle & = \langle - \Delta^2 g + f(e^{i\zeta}W_\mu +e^{i\theta} W_\lambda +g   )-   f(e^{i\zeta}W_\mu +e^{i\theta} W_\lambda ) , A_0 (\lambda )g \rangle\\
&\qquad
+o\big(|t|^{-\frac{N-11}{N-12}}\big) .
\end{align*}

Using \eqref{properA1}, \eqref{properA3} and since $A_0 (\lambda) g = \frac{2}{N\lambda^4}\Delta q \big(\frac{\cdot}{\lambda}\big) g +A(\lambda ) g$,
\begin{align*}
&\frac{d}{dt} \big\langle g, i A_0 (\lambda) g\big\rangle 
\leq \frac{2c_0}{\lambda^4} \|g\|_{\mathcal{E}}^2\\
&- \frac{1}{\lambda^4} \Big(\int_{|x|\leq R\lambda} |\Delta g|^2 dx -\big\langle f(e^{i\zeta}W_\mu +e^{i\theta} W_\lambda +g) - f(e^{i\zeta} W_\mu +e^{i\theta} W_\lambda), \frac{2}{N} \Delta q \big(\frac{\cdot}{\lambda} \big) g\big\rangle  \Big) \\
&-\big\langle A(\lambda) (e^{i\zeta} W_\mu +e^{i\theta} W_\lambda), f (e^{i\zeta} W_\mu +e^{i\theta} W_\lambda+g) - f(e^{i\zeta} W_\mu +e^{i\theta} W_\lambda)\\
&\qquad - f^\prime (e^{i\zeta} W_\mu +e^{i\theta} W_\lambda)g\big\rangle .
\end{align*}
By applying Taylor's expansion \eqref{fe3}, we get 
\begin{align*}
& \| f (e^{i\zeta}W_\mu +e^{i\theta} W_\lambda +g) - f (e^{i\zeta} W_\mu +e^{i\theta} W_\lambda )  - f^\prime (e^{i\zeta} W_\mu +e^{i\theta} W_\lambda) g \|_{L^{\frac{2N}{N+4}}}\\
& \lesssim \|f(g) \|_{L^{\frac{2N}{N+4}}}  \lesssim  \|g\|_{\mathcal{E}}^{\frac{N+4}{N-4}} \ll \|g \|_{\mathcal{E}}.
\end{align*}
From \eqref{fe1}, we also have
\begin{align*}
\|f^\prime (e^{i\zeta} W_\mu +e^{i\theta} W_\lambda) - f^\prime (e^{i\theta} W_\lambda))g \|_{L^{\frac{2N}{N+4}}(|x| \leq \tilde{R} \lambda) } & \lesssim \|f^\prime (e^{i\zeta} W_\mu)\|_{L^{\frac{N}{4}} (|x|\leq \tilde{R} \lambda)} \|g\|_{\mathcal{E}}\\
&\ll \|g\|_{\mathcal{E}}.
\end{align*}
Combining the two previous estimate and since, by Sobolev's embedding, we have $\|\frac{1}{N} \Delta q (\frac{.}{\lambda}) g \|_{L^{\frac{2N}{N-4}}} \lesssim \|g\|_{\mathcal{E}}$, we obtain
\begin{align*}
&\big|\big\langle f(e^{i\zeta}W_\mu +e^{i\theta} W_\lambda +g) - f(e^{i\zeta}W_\mu +e^{i\theta} W_\lambda) , \frac{1}{N}\Delta q \big(\frac{\cdot}{\lambda}\big) g  \big\rangle - \big\langle f^\prime (e^{i\theta}W_\lambda)g , \frac{1}{N}\Delta q \big(\frac{\cdot}{\lambda}\big) g \big\rangle\big|\\
& \ll 
\|g\|_{\mathcal{E}}^2 . 
\end{align*}

Since $\frac{1}{N}\Delta q \big(\frac{x}{\lambda}\big)=1$, for $|x|\leq R \lambda$, and $\|f^\prime (e^{i\theta} W_\lambda)\|_{L^{\frac{N}{4}} (|x|\geq R\lambda)} \ll 1$ for $R$ large enough, we have
$$\big|\big\langle f(e^{i\zeta}W_\mu +e^{i\theta} W_\lambda +g) - f(e^{i\zeta}W_\mu +e^{i\theta} W_\lambda) , \frac{1}{N}\Delta q \big(\frac{\cdot}{\lambda}\big) g  \big\rangle - \langle f^\prime (e^{i\theta}W_\lambda)g ,  g \rangle\big| \ll  |t|^{-\frac{N-3}{N-12}}.$$

By Proposition \ref{propcoer111}, \eqref{a1_a2_}, and the assumption \eqref{lastcondai}, we have, for some constant $c_3>0$ that can be chosen as small as we wish by enlarging $R$, that
$$\int_{|x|\leq R\lambda} |\Delta g|^2 dx - \langle f^\prime (e^{i\theta } W_\lambda) g ,g\rangle \geq - c_3 \|g\|_{\mathcal{E}}^2 $$
 up to some negligible terms.
Hence
\begin{align*}
\frac{2c_0}{\lambda^4} \|g\|_{\mathcal{E}}^2 & - \frac{1}{\lambda^4} \Big(\int_{|x|\leq R\lambda} |\Delta g|^2 dx \\
&\qquad -\big\langle f (e^{i\zeta}W_\mu +e^{i\theta} W_\lambda +g) - f (e^{i\zeta} W_\mu +e^{i\theta} W_\lambda ) , \frac{1}{N} \Delta q \big(\frac{\cdot}{\lambda} \big) g\big\rangle  \Big) \\
& \leq \frac{c_2}{\lambda^4} \|g\|_{\mathcal{E}}^2
\end{align*}
with a positive constant $c_2$ that can be taken as small as we wish by enlarging $R$.

Since the support of $A(\lambda)(e^{i\zeta} W_\mu)$ is contained in $|x|\leq \tilde{R} \lambda$ and 
$\|A(\lambda)(e^{i\zeta} W_\mu)\|_{L^\infty}\lesssim \lambda^{-4}$, so we have
$$\|A(\lambda)(e^{i\zeta} W_\mu)\|_{L^{\frac{2N}{N-4}}} \lesssim (\lambda^N \lambda^{-\frac{4N}{N-4} } )^{\frac{N-4}{2N}}=\lambda^{\frac{N-12}{2}}\approx |t|^{-1}.  $$

On the other hand, by \eqref{fe3}, we get
$$\|f(e^{i\zeta}W_\mu +e^{i\theta} W_\lambda +g) - f (e^{i\zeta}W_\mu +e^{i\theta} W_\lambda) - f^\prime (e^{i\zeta} W_\mu +e^{i\theta} W_\lambda) g \|_{L^{\frac{2N}{N+4}}} \lesssim \|g\|_{\mathcal{E}}^{\frac{N+4}{N-4}}. $$

Observe that
$$|A(\lambda)(e^{i\theta} W_\lambda ) - \frac{1}{\lambda^4} e^{i\theta} \Lambda W_\lambda| \lesssim \frac{W_\lambda}{\lambda^4},$$
and $A(\lambda ) W = \frac{1}{\lambda^4} \Lambda W_\lambda$ for $|x|\leq R \lambda$, so we obtain
\begin{align*}
\big|\big\langle A(\lambda) (e^{i\theta }W_\lambda)- \frac{1}{\lambda^4} e^{i\theta}\Lambda W_\lambda,&\, f(e^{i\zeta}W_\mu +e^{i\theta} W_\lambda +g)- f (e^{i\zeta} W_\mu +e^{i\theta} W_\lambda)\\
&\, - f^\prime (e^{i\zeta} W_\mu +e^{i\theta} W_\lambda)g\big\rangle\big|\\
 \lesssim \frac{1}{\lambda^4} \int_{|x|\geq R\lambda} W_\lambda\big|f(e^{i\zeta}W_\mu &+e^{i\theta} W_\lambda +g)- f (e^{i\zeta} W_\mu +e^{i\theta} W_\lambda)\\
&\, - f^\prime (e^{i\zeta} W_\mu +e^{i\theta} W_\lambda)g \big| dx.
\end{align*}
Since $|\zeta -\theta|\approx \frac{\pi}{2} $, we have $|e^{i\zeta}W_\mu + e^{i\theta}W_\lambda | \gtrsim W_\lambda $, and therefore \eqref{fe3} implies that
$$W_\lambda \big|f(e^{i\zeta}W_\mu +e^{i\theta} W_\lambda +g)- f (e^{i\zeta} W_\mu +e^{i\theta} W_\lambda) - f^\prime (e^{i\zeta} W_\mu +e^{i\theta} W_\lambda)g\big| \lesssim W_\lambda^{\frac{8}{N-4}} |g|^2 .$$
Integrating the last term over the region $|x|\geq R\lambda$, we find overall that
$$\frac{1}{2} \dfrac{d}{dt} \langle g ,iA_0 (\lambda ) g \rangle \leq \frac{c_1}{2\lambda^4 } \|g\|_{\mathcal{E}}^2 +\dfrac{K}{\lambda^4}.$$
So we proved that 
\[
\psi^\prime (t) \geq - c_1 |t|^{-\frac{N-11}{N-12}}.
\]
Since $\theta (T)=0 $, we have $|\psi (T)|\lesssim \|g\|_{\mathcal{E}}^2\ll |t|^{-\frac{1}{N-12}} $. Integration of \eqref{4-21} over $[T,t]$ yields $\psi(t) \gtrsim -c_1|t|^{-\frac{1}{N-1}}$. But, on the other hand,
$$|\langle g(t) , A_0 (\lambda ) g(t) \rangle| \lesssim \|g(t)\|_{\mathcal{E}}^2\le |t|^{-\frac{N-3}{N-12}}
\ll |t|^{-\frac{1}{N-12}}.$$
Hence $\theta (t) \gtrsim -c_1|t|^{-\frac{1}{N-12}}$ that implies \eqref{4-20} provided $c_1>0$ is chosen small enough. We have proved \eqref{lastconde3}.

Finally, from \eqref{propenere2} we get
$\|g\|_{\mathcal{E}}^2 +C_0 \theta \lambda^{\frac{N-4}{2}}\le C|t|^{-\frac{N-2}{N-12}}$, and therefore
\[
\|g\|_{\mathcal{E}}^2 \le C_0 \theta \lambda^{\frac{N-4}{2}} +C|t|^{-\frac{N-2}{N-12}}\le \frac{1}{8}|t|^{-\frac{N-3}{N-12}} +C|t|^{-\frac{N-2}{N-12}}
\]
if the constant $c_0$ in \eqref{(4-18)} is small enough.
This proves \eqref{lastconde4}. 
\end{proof}

In view of the previous proposition, it only remains to control $\lambda$, $a_1^+$ and $a_2^+$. This will be the purpose of the next proposition.
\begin{prop}
\label{Proplast}
Let $T_0<0$ with $|T_0 |$ large enough. For $T<T_0$, there exist $\lambda^0, a_1^0, a_2^0$ satisfying \eqref{condinit} such that the solution with initial data $u(T)=-iW+W_{\lambda_0}+g^0$ exists on the time interval $[T,T_0]$ and for $t\in [T,T_0]$ the bounds \eqref{lastconde1}-\eqref{lastconde4} and
\begin{equation}
\label{lasteqee1}
|\lambda (t) - \tilde{C}|t|^{-\frac{2}{N-12}}|\leq \frac{1}{2} |t|^{-\frac{5}{2(N-12)}},
\end{equation}
\begin{equation}
\label{lasteqee2}
|a_1^+ (t)|\leq \frac{1}{2} |t|^{-\frac{N}{2(N-12)}},
\end{equation}
\begin{equation}
\label{lasteqee3}
|a_2^+ (t)|\leq \frac{1}{2} |t|^{-\frac{N}{2(N-12)}}.
\end{equation}
\end{prop}

The proof will be based on the Brouwer fixed point theorem. We will need some preparation before proceeding.
For $t\in [T,T_0]$, $\tilde{\lambda} >0$, $\tilde{a}_i \in \R$, $i=1,2$, we define
$$X_t (\tilde{\lambda} , \tilde{a}_1 , \tilde{a}_2) = (\tilde{C} |t|^{-\frac{2}{N-12}} + \tilde{\lambda} |t|^{-\frac{5}{2(N-12)}}, \tilde{a}_1 |t|^{-\frac{N}{2(N-12)}}, \tilde{a}_2 |t|^{-\frac{N}{2(N-12)}} ).$$
So our parameters satisfy \eqref{lasteqee1}-\eqref{lasteqee3} if and only if
$$X^{-1}_t (\lambda (t) , a_1^+ (t) , a_2^+ (t)) \in Q =[-\frac{1}{2} , \frac{1}{2} ]^3.$$

\begin{lem}
Assume that $\lambda (t), a_1^+ (t)$ and $a_2^+ (t)$ satisfy \eqref{mode8}, \eqref{stcoe1} and \eqref{stcoe3} for $t\in (T_1 ,T_2)\subset [T,T_0]$ and that, for $t\in (T_1 ,T_2)$,
$$\big(p_0(t) ,p_1(t) ,p_2(t)\big) := X_t^{-1} (\lambda (t) , a_1^+ (t) , a_2^+ (t)) \in Q \backslash \partial Q .$$
Then, for all $t\in (T_1,T_2)$, we have
\begin{equation}
\label{blablaree1}
\left|p_0^\prime (t) - \frac{2N-25}{2(N-12)} |t|^{-1} p_0 (t)\right| \leq c |t|^{-1}
\end{equation}
\begin{equation}
\label{blablaree2}
|p_1^\prime (t) - \frac{\nu}{\mu (t)^4} p_1 (t)| \leq \frac{c}{\mu (t)^4}, 
\end{equation}
\begin{equation}
\label{blablaree3}
|p_2^\prime (t) - \frac{\nu}{\lambda (t)^4 } p_2 (t)| \leq \frac{c}{\lambda (t)^4}, 
\end{equation}
where $c>0$ is as small as we want provided that $|T_0|$ is large enough.
\end{lem}

\begin{proof}
By definition, we have
$$\lambda (t)= \tilde{C} |t|^{-\frac{2}{N-12}}+ p_0 (t) |t|^{-\frac{5}{2(N-12)}}.$$
So, taking the derivative with respect to $t$, we find
$$\lambda^\prime (t)= \frac{2\tilde{C}}{N-12} |t|^{-\frac{N-10}{N-12}} + \frac{5p_0 (t)}{2(N-12)} |t|^{-\frac{2N-19}{2(N-12)}} + p_0^\prime (t)|t|^{-\frac{5}{2(N-12)}}. $$
By the Newton's binomial formula and since $|p_0| \lesssim 1$, we have
$$\lambda (t)^{\frac{N-10}{2}}= \tilde{C}^{\frac{N-10}{2}} |t|^{-\frac{N-10}{N-12}} + \frac{N-10}{2} \tilde{C}^{\frac{N-12}{2}} |t|^{-1} p_0 (t) |t|^{-\frac{5}{2(N-12)}}+ O(|t|^{-\frac{N-9}{N-12}}). $$
We want to apply Proposition \ref{lem3.1} and therefore we choose $\tilde{C}$ such that 
\begin{equation}\label{defc}
\frac{2\tilde{C}}{N-12}= \frac{C_1 \tilde{C}^{\frac{N-10}{2}}}{2\|W\|_{L^2}^2}.
\end{equation}

We obtain 
\begin{align*}
\lambda^\prime (t) - \frac{C_1}{2 \|W\|_{L^2}^2} \lambda (t)^{\frac{N-10}{2}} &
 = \left(\frac{5}{2(N-12)} - \frac{(N-10)C_1 \tilde{C}^{\frac{N-12}{2}}}{4 \|W\|_{L^2}^2} \right)p_0 (t) |t|^{-\frac{2N-19}{2(N-12)}} \\
 &\ + p_0^\prime (t)|t|^{-\frac{5}{2(N-12)}} + O\big(|t|^{-\frac{N-9}{N-12}}\big). 
\end{align*}
Multiplying both sides by $|t|^{\tfrac{5}{2(N-12)}}$ and applying \eqref{mode8}
we obtain \eqref{blablaree1}.

We have $a_2^+ (t)=  p_2 (t)|t|^{-\tfrac{N}{2(N-12)}}$, so
$$\frac{d}{dt}a_2^+(t) - \frac{\nu}{\lambda(t)^4}a_2^+(t) = |t|^{- \frac{N}{2(N-12)}} \left(p_2^\prime (t) - \frac{\nu}{\lambda(t)^4} p_2 (t)\right) + O\big(|t|^{-\frac{N}{2 (N-12)} -1}\big).$$
Hence
\begin{align*}
\left|p_2^\prime (t)-\frac{\nu}{\lambda(t)^4}\right|
&=\left|\frac{d}{dt}a_2^+(t)-\frac{\nu}{\lambda(t)^4}\right||t|^{\frac{N}{2(N-12)}} +O\big(|t|^{-1}\big)\\
&\le \frac{c}{\lambda(t)^4}+O\big(|t|^{-1}\big)
\end{align*}
by \eqref{stcoe3}. The estimate \eqref{blablaree2} follows similarly.
\end{proof}

From \eqref{defc} we see that 
\[
\tilde{C}=\left(\frac{4\|W\|_{L^2}^2}{(N-12)C_1}\right)^{\frac{2}{N-12}},
\]
where
\[
C_1=\int_{\R^N} W^{\frac{N+4}{N-4}} dx.
\]
justifying the definition \eqref{1deftildec}.

For $C>1$, $j=0,1,2$ and $p=(p_0,p_1,p_2)\in \R^3$, we define
$$V_j (C,p)=\{p+(r_0 ,r_1,r_2): \sign (r_j)= \sign(p_j)\ \text{and} \max_k |r_k|<C|r_j|\}.$$

\begin{lem}\label{4-10}
Assume that $\lambda (t) , a_1^+ (t)$ and $a_2^+ (t)$ satisfy \eqref{mode4}, \eqref{mode8}, \eqref{stcoe1}, \eqref{stcoe3}, and \eqref{lastcondai} for $t\in (T_1 ,T_2)$. There exists a constant $C>0$ depending on $T_1$ and $T_2$ such that if for some $T_3 \in (T_1 ,T_2))$ and $j=0,1,2$, we have $|p_j (T_3)| \geq \frac{1}{4}$, then for all $t\in (T_3,T_2)$ we have $p(t) \in V_j (C, p(T_3))$.
\end{lem}

\begin{proof}
Using the previous lemma, we see that there exist constants $c_1,C_1>0$ depending on $T_1$ and $T_2$ such that $|p_j^\prime (t)|\leq C_1$ and $|p_j (t)|\geq \frac{1}{4}$ implies that $|p_j^\prime (t)|\geq c_1$ and sign $p_j^\prime (t)=$ sign $p_j (t)$. Then it suffices to take $C>\frac{C_1}{c_1}$. 
\end{proof}

\begin{proof}[Proof of Proposition \ref{Proplast}]
Suppose that the result does not hold. We will construct a continuous retraction $\Phi : Q \rightarrow \partial Q$, $\Phi (p)=p$, for $p\in \partial Q$. However, this leads to a contradiction since by the Brouwer fixed point theorem 
such a map cannot exist.

Let $p^0 \in Q$. Take $(\lambda^0 , {a}_1^0 , {a}_2^0 )= X_T (p^0)$ and let $g^0$ be given by Lemma \ref{lemcondinit}. Let $u: [T,T_+) \rightarrow \mathcal{E}$ be the solution of \eqref{eq} for the initial data $u(T)=-iW +W_{\lambda^0}+g^0$. We will say that the solution $u$ is associated with $p^0 \in Q$.

Let $T_2$ be the infimum of the values of $t\in [T,T_+)$ such that at least one of the conditions \eqref{lastconde1}-\eqref{lastconde4}, \eqref{lasteqee1}, \eqref{lasteqee2} or \eqref{lasteqee3} does not hold. By counter assumption Proposition \ref{Proplast} is false, so $T_2$  exists and $T_2<T_0$. Indeed, if all the conditions were true for $t\in [T,T_+)$, then $T_+>T_0$ and hence all the conditions would hold on $[T,T_0]$.

Set $p^1 = X_{T_2}^{-1} (\lambda (T_2) , a_1^+ (T_2) ,a_2^+ (T_2))$. By continuity $p^1 \in Q$ and we will show that $p^1 \in \partial Q$. By continuity of the flow, the assumptions of Proposition \ref{Proplast} are satisfied for $T_1 =T_2  +\tau$ for some $\tau>0$. Hence \eqref{lastconde1}-\eqref{lastconde4} hold on $[T_2 , T_2 +\tau]$, so at least one of the conditions \eqref{lasteqee1}, \eqref{lasteqee2} or \eqref{lasteqee3} is false somewhere on $[T_2 , T_2+\tau]$ for every $\tau>0$. By continuity of the parameters with respect to time, this yields that $p^1 \in \partial Q$.

 We set
$$\Phi : Q \rightarrow \partial Q,\ \Phi (p^0)=p^1.$$
We see from the definition that $\Phi (p)=p$ for $p\in \partial Q$. Next, let us show that $\Phi$ is continuous.

Let $p^0 \in Q$, $\Phi (p^0) =p^1 \in \partial Q$ and $\varepsilon >0$. Let $C$ be the constant defined in Lemma \ref{4-10} for $T_1=T$ and $T_2=T_0$. Without loss of generality, we suppose that $p_0^1 = \frac{1}{2}$. For $\delta>0$ small enough, $V_\delta = V_0 (C, \frac{1}{2}-\delta , p_1^1 ,p_2^1)\cap \partial Q$ is an $\varepsilon$-neighborhood of $p^1$. So, from Lemma \ref{4-10}, we only need to show that if $q^0 \in Q$ with $|q^0 -p^0|$ small enough, then the solution associated with $q^0$ passes through $V_\delta$.

If $p^0=p^1 \in \partial Q$, this holds since $V_\delta$ is a neighborhood of $p^0$. If $p^0 \in Q \backslash \partial Q$, the solution associated with $p^0$ passes through $V_\delta$ before reaching $\partial Q$. Thus by continuous dependence on the initial data the solution associated with $q^0$ has to pass through $V_\delta$ if $|q^0 -p^0|$ is small enough.

\end{proof}
We are finally in position to prove Theorem \ref{main}.

\begin{proof}[Proof of Theorem \ref{main}]
Let $T_0<0$ be given by Proposition \ref{Proplast} and let $T_0,T_1,T_2,\ldots$ be a decreasing sequence tending to $-\infty$. For $n\geq 1$, let $u_n$ be the solution given by Proposition \ref{Proplast}. Then \eqref{lastconde1}-\eqref{lastconde4}, and \eqref{lasteqee1} yield
\begin{equation}
\label{verylasteq}
\|u_n (t) - (-iW +W_{\tilde{C} |t|^{-2/(N-12)}})\|_{\mathcal{E}}\lesssim |t|^{-\frac{1}{2(N-12)}},
\end{equation}
for all $t\in [T_n ,T_0]$ and with a constant independent of $n$. Passing to a subsequence if necessary, we can assume that $u_n (T_0) \rightharpoonup u_0 \in \mathcal{E}$. Let $u$ be the solution of \eqref{eq} with initial condition $u(T_0)=u_0$. Using nonlinear profile decomposition (see \cite{MR3462127} for the linear one and we refer to \cite{MR2781926} Proposition $2.8$ and  \cite{MR2654778} proof of Lemma $3.2$ for the nonlinear one), we can prove proceeding as \cite{MR3904767} that there exists a constant $\eta>0$ such that the following holds. Let $K\subset \mathcal{E}$ be a compact set and let $u_n : [T_1, T_2] \rightarrow \mathcal{E}$ be a sequence of solutions of \eqref{eq} such that
$$\text{dist} (u_n (t) , K)\leq \eta , \ \text{for all } n\in \N \ \text{and}\ t\in [T_1 ,T_2].$$
Suppose that $u_n (T_1) \rightharpoonup u_0 \in \mathcal{E}$. Then the solution $u(t)$ of \eqref{eq} with initial data $u(T_1)=u_0$ is defined for $t\in [T_1 ,T_2]$ and 
$$u_n (t) \rightharpoonup u(t)\ \text{ for all}\  t\in [T_1 ,T_2].$$
Applying the previous result, we see that $u$ exists on the time interval $(-\infty , T_0]$ and for all $t\in (-\infty ,T_0]$, we have $u_n (t) \rightharpoonup u(t)$. Passing to the weak limit in \eqref{verylasteq}, we finish the proof.
\end{proof}




\end{document}